\documentclass[12pt,leqno]{amsart}
\usepackage{hyperref}
\usepackage{amsmath,amsthm,amscd,amssymb}
\usepackage{amsfonts}
\usepackage{enumerate}

\newcommand{\arxiv}[1]{\href{http://arxiv.org/#1}{arXiv:#1}}
\newcommand*{\mailto}[1]{\href{mailto:#1}{\nolinkurl{#1}}}

\newtheorem{theorem}{Theorem}[section]
\newtheorem{lemma}[theorem]{Lemma}

\numberwithin{equation}{section}

\newtheorem{re}{Remark}[section]
\newtheorem{prop}{Proposition}[section]
\newtheorem{theo}{Theorem}[section]
\newtheorem{lem}{Lemma}[section]
\newtheorem{col}{Corollary}[section]

\newcommand{\be}{\begin{equation}}
\newcommand{\ee}{\end{equation}}
\newcommand\bes{\begin{eqnarray}} \newcommand\ees{\end{eqnarray}}
\newcommand{\bess}{\begin{eqnarray*}}
\newcommand{\eess}{\end{eqnarray*}}
\newcommand{\D}{\displaystyle}

\newcommand{\bu}{{\bf u}}
\newcommand{\bH}{{\bf H}}

\newcommand{\bF}{{\bf F}}
\newcommand{\bI}{{\bf I}}

\newcommand{\equ}{\overset{\rm def}{=}}

\def\XXint#1#2#3{{\setbox0=\hbox{$#1{#2#3}{\int}$}
     \vcenter{\hbox{$#2#3$}}\kern-.5\wd0}}

\numberwithin{equation}{section}

\begin{document}

\title[Large time behavior  for  compressible viscoelastic flows]{Large time behavior in critical $L^p$ Besov spaces for compressible viscoelastic flows}

\author[Q. Bie]{Qunyi Bie}
\address[Q. Bie]{College of Science $\&$  Three Gorges Mathematical Research Center, China Three Gorges University, Yichang 443002, PR China}
\email{\mailto{bieqy@mail2.sysu.edu.cn}}

\author[H. Fang]{Hui Fang}
\address[H. Fang]{College of Science, China Three Gorges University, Yichang 443002, PR China}
\email{\mailto{1624196902@qq.com}}
\author[Q. Wang]{Qiru Wang}
\address[Q. Wang]{School of Mathematics, Sun Yat-Sen University, Guangzhou 510275, PR China}
\email{\mailto{mcswqr@mail.sysu.edu.cn}}
\author[Z.-A. Yao]{Zheng-an Yao}
\address[Z.-A. Yao]{School of Mathematics, Sun Yat-Sen University, Guangzhou 510275, PR China}
\email{\mailto{mcsyao@mail.sysu.edu.cn}}


\begin{abstract}
We consider the large time behavior of global strong solutions to the compressible viscoelastic flows on the whole space $\mathbb{R}^N\,(N\geq 2)$, where the system describes the elastic properties of the compressible fluid.  Adding a suitable initial condition involving only the low-frequency, we prove optimal time decay estimates for the global solutions in the $L^p$ critical regularity framework, which are similar to those of the
compressible Navier-Stokes equations. Our results rely on the pure energy argument, which allows us
to remove the usual smallness assumption of the data in the low-frequency.
\end{abstract}

\maketitle

\section{Introduction and main results} \label{s:1}
The  compressible  viscoelastic flows in $ \mathbb{R}_+\times\mathbb{R}^N\ $ reads as
 \begin{equation}\label{1.1}
 \left\{\begin{array}{ll}\medskip\displaystyle\partial_t \rho+{\rm div}(\rho {\bf u})
=0,\\
 \medskip\displaystyle\partial_t(\rho{\bf u})+{\rm div}(\rho{\bf u}\otimes{\bf u})+\nabla P(\rho)={\rm div}(2\mu D(\bu)+\lambda {\rm div}\bu\, {\rm \bI})+\alpha{\rm div}(\rho \bF\bF^T ),\\
 \medskip\displaystyle\partial_t\bF+\bu\cdot\nabla\bF=\nabla\bu\bF,\\
 \D (\rho,\bF, \bu)(0)=(\rho_0, \bF_0, \bu_0),
 \end{array}
 \right.
 \end{equation}
where $\rho=\rho(t,x)\in\mathbb{R}_+$ and $\bu=\bu(t,x)\in \mathbb{R}^N (N\geq 2)$ represent the density and velocity field, respectively, and $\bF\in\mathbb{R}^{N\times N}$ is the deformation gradient. Here $\bF^T$ means the transpose matrix of $\bF$, and $\bI$  is the unit matrix. The pressure $P$ depends only on the density and the function will be taken suitably smooth. The notation $D(\bu)\equ\frac{1}{2}(D_x\bu + (D_x\bu)^T)$ stands
for the deformation tensor. The Lam\'{e} coefficients $\lambda$ and $\mu$ (the bulk and shear viscosities) are density-dependent functions, which are supposed to be smooth functions of density and to satisfy $\mu>0$ and $\lambda+2\mu>0$. Such a condition ensures ellipticity for the operator ${\rm div}(2\mu D(\bu)+\lambda {\rm div}\bu\, {\rm Id})$ and is satisfied in the physical cases. Let us mention that we focus on solutions $(\rho, {\bf F}, \bu)$ that are close to some constant state $(1, {\bf I}, {\bf 0})$, at spatial infinity.

The main purpose of this paper is to investigate the time decay rates of strong solutions to system \eqref{1.1} in the critical $L^p$ framework.  Let us note that system \eqref{1.1} is scaling invariant under the transformation: for any constant $\kappa>0$,
\be\nonumber
\tilde{\rho}=\rho(\kappa^2t, \kappa x),~~\tilde{\bf F}={\bf F}(\kappa^2t, \kappa x),~~\tilde{\bf u}=\kappa{\bf u}(\kappa^2t, \kappa x)
\ee
 up to changes of the pressure  $\tilde{P}=\kappa^2P$. Here a functional space is called a critical space if the associated  norm is invariant under the scaling
$$
(\tilde{e},\tilde{\bf f}, \tilde{\bf g})(x)=(e(\kappa x), {\bf f}(\kappa x), \kappa{\bf g}(\kappa x)).
$$

 Let us first recall some local and global existence results for the compressible viscoelastic flows. Lei and Zhou \cite{lei2005global} proved the global existence of classical solutions for the 2D model by the incompressible limit. Hu and Wang \cite{hu2010local} obtained the local existence of strong solutions. Hu and Wu \cite{hu2013globalexistence} proved the global existence of strong solutions to \eqref{1.1} as initial data are the small perturbation of $(1, {\bf I}, {\bf 0})$ in $H^2(\mathbb{R}^3)$. In addition, with the extra $L^1(\mathbb{R}^3)$ assumption, the
optimal convergence rates of the solutions in $L^p$-norm with $2\leq p\leq 6$ and optimal convergence rates
of their spatial derivatives in $L^2$-norm were obtained. Hu and Wang \cite{hu2011global} and Qian and Zhang \cite{qian2010global} independently derived the global existence with initial data near equilibrium in the critical $L^2$ space. Very recently, Pan and Xu \cite{pan2019global} extended the works \cite{hu2011global, qian2010global} to the critical $L^p$ Besov space and further obtained the optimal time decay estimates of strong solutions in the general  $L^p$ critical framework. As for the incompressible viscoelastic flows, one could refer to the works \cite{chemin2001about, lin2005on, lei2008global, qian2010well, zhang2012global} and the references therein.


 In the case ${\bf F}\equiv {\bf 0}$, system \eqref{1.1} reduces to the  classical compressible Navier-Stokes equations. In the critical framework, for the compressible or incompressible Navier-Stokes system,  there have been a lot of results,  see for example \cite{cannone1997a,  charve2010global,  chen2010global, danchin2000global, danchin2016fourier, danchin2016incompressible, danchin2016optimal, fujita1964on,  haspot2011existence, kozono1994semilinear,  okita2014optimal,  xin2018optimal, xu2019a}. In particular, concerning the large time asymptotic behavior of strong solutions for the compressible Navier-Stokes equations in the critical framework, Okita \cite{okita2014optimal} performed low and high frequency decompositions and proved the time decay rate for strong solutions in the $L^2$  critical framework and in dimension $N\geq 3$.
 Danchin in the survey paper \cite{danchin2016fourier} proposed another description of the time decay which enables to proceed with dimension $N\geq 2$ in the $L^2$ critical framework. Recently,  Danchin and Xu \cite{danchin2016optimal} extended the method of \cite{danchin2016fourier} to derive optimal time decay rate in the
 general $L^p$ type critical spaces. Later on, depending on the refined time-weighted energy approach in the Fourier semi-group framework, Xu \cite{xu2019a} developed a general low-frequency condition for optimal decay estimates, where the regularity index $\sigma_1$ of $\dot{B}_{2,\infty}^{-\sigma_1}$
 belongs to a whole range $(1-\frac{N}{2}, \frac{2N}{p}-\frac{N}{2}]$.  Very recently, inspired by the ideas in \cite{guo2012decay, strain2006almost}, Xin and Xu \cite{xin2018optimal} developed a new energy  argument to remove the usual smallness condition of low frequencies studied in \cite{danchin2016optimal}.

 In this paper, motivated  by the works \cite{guo2012decay, pan2019global, strain2006almost, xin2018optimal}, we are going to establish the optimal decay for system \eqref{1.1} in the $L^p$ type critical framework without the smallness assumption of low frequencies. Now, let us first recall the global existence result of system \eqref{1.1} in the  critical $L^p$ framework (see \cite{pan2019global}).

\begin{theo}\label{th1.1} {\rm(}\cite{pan2019global}{\rm )}
Let $N\geq 2$ and $p$ satisfy
\be\label{1.100}
2\leq p\leq \min\{4, {2N}/(N-2)\}\,\,\,{\rm and},\,\,\,additionally,\,\,\,p\neq 4\,\,\,{\rm if}\,\,\,N=2.
\ee
Assume that $P^\prime(1)>0$. There exists a small positive constant $c=c(p, N, \lambda, \mu, P)$ and a universal integer $k_0\in \mathbb{Z}$ such that if \,$b_0\equ \rho_0-1\in \dot{B}
_{p,1}^{\frac{N}{p}},\, \bH_0\equ \bF_0-\bI\in \dot{B}
_{p,1}^{\frac{N}{p}},\,\bu_0\in \dot{B}
_{p,1}^{\frac{N}{p}-1}$ and if in addition $(b_0^\ell, \bH_0^\ell, \bu_0^\ell)\in\dot{B}
_{2,1}^{\frac{N}{2}-1}$ {\rm(} with the notation $z^\ell\equ \dot{S}_{k_0+1}z$ and $z^h=z-z^\ell$ {\rm)} with
\be\label{1.5}
\mathcal{X}_{p,0}\equ \|(b_0, \bH_0, \bu_0)\|_{\dot{B}_{2,1}^{\frac{N}{2}-1}}^\ell
+\|(b_0, \bH_0)\|_{\dot{B}_{p,1}^{\frac{N}{p}}}^{h}
+\|{\bf u}_0\|_{\dot{B}_{p,1}^{\frac{N}{p}-1}}^{h}\leq c,
\ee
then \eqref{1.1} has a unique global solution $(\rho, \bF, \bu)$ with $\rho=b+1, \bF=\bH+\bI$ and
 $(\rho, \bF, \bu)$ in the space $X_{p}$ defined by
 \begin{equation}\nonumber
 \left.\begin{array}{ll}\medskip\D
b^{\ell}\in\widetilde{\mathcal{C}}_b(\mathbb{R}_+;\dot{B}_{2,1}^{\frac{N}{2}-1})\cap L^1(\mathbb{R}_+;\dot{B}_{2,1}^{\frac{N}{2}+1}),\,\,\,\,\,\,
\bH^{\ell}\in\widetilde{\mathcal{C}}_b(\mathbb{R}_+;\dot{B}_{2,1}^{\frac{N}{2}-1})\cap L^1(\mathbb{R}_+;\dot{B}_{2,1}^{\frac{N}{2}+1}),\\\medskip\D
{\bu}^{\ell}\in\widetilde{\mathcal{C}}_b(\mathbb{R}_+;\dot{B}_{2,1}^{\frac{N}{2}-1})\cap L^1(\mathbb{R}_+;\dot{B}_{2,1}^{\frac{N}{2}+1}),\,\,\,\,\,\,b^{h}\in\widetilde{\mathcal{C}}_b(\mathbb{R}_+;\dot{B}_{p,1}^{\frac{N}{p}})\cap L^1(\mathbb{R}_+;\dot{B}_{p,1}^{\frac{N}{p}}),\\\medskip\D
  \bH^{h}\in\widetilde{\mathcal{C}}_b(\mathbb{R}_+;\dot{B}_{p,1}^{\frac{N}{p}})\cap L^1(\mathbb{R}_+;\dot{B}_{p,1}^{\frac{N}{p}}),\,\,\,\,\,\,\,\,\,\,\,\,\,{\bf u}^{h}\in\widetilde{\mathcal{C}}_b(\mathbb{R}_+;\dot{B}_{p,1}^{\frac{N}{p}-1})\cap L^1(\mathbb{R}_+;\dot{B}_{p,1}^{\frac{N}{p}+1}).
  \end{array}
  \right.
  \end{equation}
Moreover, we have for some constant $C=C(p, N, \lambda, \mu, P)$ and for any $t>0$,
\be\label{1.101}
\mathcal{X}_p(t)\leq C \mathcal{X}_{p,0},
\ee
with
\be\label{1.102}
\begin{split}
\mathcal{X}_p(t)&\equ \|(b, \bH, \bu)\|_{\widetilde{L}_t^\infty(\dot{B}_{2,1}^{\frac{N}{2}-1})\cap
{L}_t^1(\dot{B}_{2,1}^{\frac{N}{2}+1})}^\ell\\[1ex]
&\quad\quad\quad\quad\quad\quad\quad\quad
+\|(b, \bH)\|_{\widetilde{L}_t^\infty(\dot{B}_{p,1}^{\frac{N}{p}})\cap
{L}_t^1(\dot{B}_{p,1}^{\frac{N}{p}})}^h+\|\bu\|_{\widetilde{L}_t
^\infty(\dot{B}_{p,1}^{\frac{N}{p}-1})\cap
{L}_t^1(\dot{B}_{p,1}^{\frac{N}{p}+1})}^h.
\end{split}
\ee
\end{theo}

To exhibit the large-time asymptotic description of the constructed solution in Theorem \eqref{th1.1}, it is convenient to rewrite \eqref{1.1} as the nonlinear perturbation form of constant equilibrium $(1, \bI, {\bf 0})$, looking at the nonlinearities as source terms. For simplicity, we assume that $P^\prime(1)=1$. After changing the functions as
$$
b=\rho-1,\,\,\,\,\,\,\bH=\bF-\bI,
$$
we see that system \eqref{1.1} becomes
\begin{equation}\label{2.2}
\left\{\begin{array}{ll}
\partial_tb+{\rm div}\bu=-{\rm div}(b\bu),\\[1ex]
\partial_tu^i-{\mathcal{A}}u^i+\partial_i b-\alpha\partial_kH^{ik}=-\bu\cdot\nabla u^i-I(b){\mathcal{A}}u^i-K(b)\partial_i b\\
\quad
+\alpha H^{jk}\partial_j H^{ik}+\displaystyle\frac{1}{1+b}\Big({\rm div}\big(2\widetilde{\mu}(b)D(\bu)+\widetilde{\lambda}(b){\rm div}\bu\,{\bI}\big)\Big)^i,\\[2ex]
\partial_t\bH-\nabla\bu=\nabla\bu \bH-\bu\cdot\nabla\bH,\\[1ex]
(b, \bH, \bu)|_{t=0}=(b_0, \bH_0, \bu_0),
\end{array}
\right.
\end{equation}
with
{\small\begin{equation}\nonumber
\begin{array}{ll}\displaystyle
I(b)\equ \frac{b}{1+b},\,\,\,\,\,\,K(b)\equ \frac{P^\prime(1+b)}
{1+b}-1,\,\,\,
\widetilde{\mu}(b)\equ \mu(1+b)-\mu(1),\,\,\,\widetilde{\lambda}(b)\equ\lambda(1+b)-\lambda(1),\\[2ex]
\mathcal{A}\equ \bar{\mu}\Delta+(\bar{\lambda}+\bar{\mu})\nabla{\rm div}\,\,\,{\rm such\,\, that}\,\,\,2\bar{\mu}+\bar{\lambda}=1\,\,\,{\rm and}\,\,\,\bar{\mu}>0 \,\, (\bar{\mu}\equ\mu(1)\,\,\,{\rm and}\,\,\,\bar{\lambda}\equ\lambda(1)).
\end{array}
\end{equation}}

Let us emphasize that, in the higher order Sobolev spaces,  Hu and Wang \cite{hu2013globalexistence} investigated the optimal time decay rates of global solutions to system \eqref{1.1}.  While in the critical $L^p$-type framework, Pan and Xu \cite{pan2019global} also studied their large-time behavior. In this paper, we will further studied the large-time behavior of global strong solutions to \eqref{1.1} in the $L^p$-type critical Besov space, and our results could be seen as the complement of the ones in \cite{pan2019global} (see Remarks \ref{re1} and \ref{re2} below).

Now, we state the main results of this paper as follows.
\begin{theo}\label{th2}
Let $N\geq 2$ and $p$ satisfy assumption \eqref{1.100}. Let $(\rho, \bF, \bu)$ be the global solution addressed by Theorem \ref{th1.1}. If in addition
$(b_0, \bH_0, \bu_0)^\ell\in \dot{B}_{2,\infty}^{-\sigma_1}\,(1-\frac{N}{2}<\sigma_1\leq \sigma_0
\equ \frac{2N}{p}-\frac{N}{2})$ such that $\|(b_0, \bH_0, \bu_0)\|_{\dot{B}_{2,\infty}^{-\sigma_1}}^\ell$
is bounded, then we have
\begin{equation}\label{1}
\|(b, \bH, \bu)(t)\|_{\dot{B}_{p,1}^{\sigma}}\lesssim (1+t)^{-\frac{N}{2}(\frac{1}{2}-\frac{1}{p})-\frac{\sigma+\sigma_1}{2}}
\end{equation}
where $-\sigma_1
-\frac{N}{2}+\frac{N}{p}<\sigma\leq\frac{N}{p}-1$ for all $t\geq 0$.
\end{theo}

Denote $\Lambda^sf\equ \mathcal{F}^{-1}(|\xi|^s\mathcal{F}f)$ for $s\in\mathbb{R}$. We would obtain the following $\dot{B}_{2,\infty}^{-\sigma_1}$-$L^r$ type decay estimates by using improved Gagliardo-Nirenberg inequalities.
\begin{col}\label{col1}
Let those assumptions of Theorem \ref{th2} be fulfilled. Then the corresponding solution $(b, \bH, \bu)$ admits
\begin{equation}\nonumber
\|\Lambda^l (b, \bH, \bu)\|_{L^r}\lesssim (1+t)^{-\frac{N}{2}(\frac{1}{2}-\frac{1}{r})-\frac{l+\sigma_1}{2}},
\end{equation}
where  $-\sigma_1
-\frac{N}{2}+\frac{N}{p}<l+\frac{N}{p}-\frac{N}{r}\leq \frac{N}{p}-1$ for $p\leq r\leq \infty$ and $t\geq 0$.
\end{col}

We give some comments as follows.

\begin{re}\label{re1}
{\rm The low-frequency assumption of initial data in \cite{pan2019global} is at the endpoint $\sigma_0$ and the corresponding norm needs to be small enough, i.e., there exists a positive constant $c=c(p, N, \mu, \lambda, P)$ such that
$
\|(b_0,\bH_0,\bu_0)\|_{\dot{B}_{2,\infty}^{-\sigma_0}}^\ell\leq c\,\,\,{\rm with}\,\,\,\sigma_0\equ \frac{2N}{p}-\frac{N}{2}.
$
Here, the new lower bound $1-\frac{N}{2}<\sigma_1\leq \sigma_0
$ enables us to enjoy larger freedom on the choice of $\sigma_1$, which allows to obtain more optimal decay estimates in the $L^p$ framework. In addition, the smallness of low frequencies is no longer needed in Theorem \ref{th2} and Corollary \ref{col1}.}
\end{re}
\begin{re}\label{re2}
 {\rm   In \cite{pan2019global}, there is a little loss on decay rates  due to the use of different Sobolev embeddings at low (or high) frequencies. For example, when $\sigma_1=\sigma_0$, the result in \cite{pan2019global} presents that the solution itself decays to equilibrium in $L^p$ norm with the rate of $O(t^{-\frac{N}{p}+\frac{N}{4}})$, which is no faster than that of  $O(t^{-\frac{N}{2p}})$ derived from  Corollary \ref{col1} above. }
 \end{re}
 \begin{re}
 {\rm To illustrate the decay rates in Corollary \ref{col1} are optimal, we are now in a position to exhibit  the decay rates of the heat kernel
\begin{equation}\label{1.9}
E(t)U_0\equ e^{-t\Delta}U_0.
\end{equation}
Taking the Fourier transform of \eqref{1.9} yields
\begin{equation}\nonumber
\mathcal{F}(E(t)U_0)(\xi)=e^{-t|\xi|^2}\mathcal{F}U_0(\xi).
\end{equation}
It follows from Hausdorff-Young and H\"{o}lder inequalities that
$$
\|E(t)U_0\|_{L^p}\leq \|\mathcal{F}(E(t)U_0)(\xi)\|_{L^{p^\prime}}\leq \|e^{-t|\xi|^2}\|_{L^r}\|\mathcal{F}U_0(\xi)\|_{L^{q^\prime}}\leq\|U_0\|_{L^q}t^{-\frac{N}{2r}},
$$
where $\frac{1}{p}+\frac{1}{p^\prime}=\frac{1}{q}+\frac{1}{q^\prime}=1$, $\frac{1}{p^\prime}
=\frac{1}{q^\prime}+\frac{1}{r}$, $p\geq 2$ and $1\leq q\leq 2$. Therefore, we obtain $r=p$ if choosing $q=\frac{p}{2}$, i.e., the heat kernel  has the time-decay rate of $O(t^{-\frac{N}{2p}})$ in $L^p$ norm if $U_0\in L^{\frac{p}{2}}$. Note that the embedding $L^{\frac{p}{2}}\hookrightarrow \dot{B}_{2,\infty}^{-\sigma_0}$.  One has the global solution of \eqref{2.2} decays to the constant equilibrium with the same rate if taking the endpoint regularity $\sigma_1=\sigma_0$. Thus, those decay rates in Corollary \ref{col1} are optimal.}
 \end{re}

 \begin{re}\label{re3}
   {\rm  The condition \eqref{1.100} may allow  us to consider the case $p>N$, so that the regularity index $\frac{N}{p}-1$ of $\bu$ becomes negative when $N=2, 3$. Our results are thus suitable for large highly oscillating initial velocities
   (see \cite{charve2010global, chen2010global} for more details). }
 \end{re}

 \begin{re}\label{re4}
{\rm As pointed out in \cite{xin2018optimal}, the nonlinear estimates in the low frequencies  play a key role in proving Theorem \ref{th2}. They employed  different Sobolev embeddings and interpolations  to handle the nonlinear terms in the non oscillation case $(2\leq p \leq N)$ and the oscillation case $(p>N)$, respectively. Here,  we develop a non-classical product estimate in the low frequencies (see \eqref{4.2} below), which enables us to unify  the estimates in the non oscillation case and the oscillation one.}
 \end{re}

 The rest of this paper is arranged as follows. In  Section \ref{s:3}, we first review some basic properties of homogeneous Besov spaces and give some classical and non-classical product estimates in Besov spaces.  Section \ref{s:4} is devoted to estimating $L^2$-type Besov norms at low frequencies, which plays an important role in deriving the Lyapunov-type inequality for energy norms. Section \ref{s:5} presents the proofs of Theorem \ref{th2} and Corollary \ref{col1}.
\section{Preliminaries}\label{s:3}
\setcounter{equation}{0}\setcounter{section}{2}\indent
Throughout the paper, $C$ stands for a harmless ``constant", and we sometimes write $A\lesssim B$ as an equivalent to $A\leq CB$. The notation $A\approx B$ means that $A\lesssim B$ and $B\lesssim A$. For any Banach space $X$ and $u, v\in X$, we agree that $\|(u,v)\|_X\equ \|u\|_X+\|v\|_X$. For  $p\in [1,+\infty]$ and $T>0$, the notation $L^p(0,T;X)$ or $L^p_T(X)$
denotes the set of measurable functions $f:[0,T]\rightarrow X$ with $t\mapsto
\|f(t)\|_X$ in $L^p(0,T)$, endowed with the norm
$
\|f\|_{L^p_T(X)}\equ\bigl{\|}\|f\|_X\bigr{\|}_{L^p(0,T)}.
$
We denote by $\mathcal{C}([0,T];X)$  the set of continuous functions from
$[0,T]$ to $X$.

We first recall the definition of homogeneous Besov spaces, which could be defined by using a dyadic partition of unity in Fourier variables called homogeneous Littlewood-Paley decomposition. Then  the product estimates in homogeneous Besov spaces are presented.
\subsection{Homogeneous Besov spaces}
 Choose a radial function $\varphi\in \mathcal{S}(\mathbb{R}^N)$ supported in $\mathcal{C}=\{\xi\in\mathbb{R}^N, \frac{3}{4}\leq |\xi|\leq \frac{8}{3}\}$ such that
$
\sum_{j\in\mathbb{Z}}\varphi(2^{-j}\xi)=1\quad\!\!{\rm if}\quad\!\!\xi\neq 0.
$
The homogeneous frequency localization operator $\dot{\Delta}_j$ and $\dot{S}_j$ are defined by
$$
\dot{\Delta}_j u=\varphi (2^{-j}D)u, \quad\,\dot{S}_j u=\sum_{k\leq j-1}\dot{\Delta}_k u\quad\,{\rm for}\quad\,j\in\mathbb{Z}.
$$
From this expression, we could see that
\be\label{1.8}
\dot{\Delta}_j\dot{\Delta}_kf=0\,\,\,{\rm if}\,\,\,|j-k|\geq 2,\,\,\,\,\,{\rm and}\,\,\,\,\dot{\Delta}_j(\dot{S}_{k-1}\dot{\Delta}_kf)=0
\,\,\,{\rm if}\,\,\,|j-k|\geq 5.
\ee

Let us denote the space $\mathcal{Y}^\prime(\mathbb{R}^N)$ by the quotient space of $\mathcal{S}^\prime(\mathbb{R}^N)/\mathcal{P}$ with the polynomials space $\mathcal{P}$. The formal equality
$
u=\sum_{k\in\mathbb{Z}}\dot{\Delta}_k u
$
holds true for $u\in \mathcal{Y}^\prime(\mathbb{R}^N)$ and is called the homogeneous Littlewood-Paley decomposition.

We then define the homogeneous Besov space as
$$
\dot{B}_{p,r}^s={\Big\{}f\in\mathcal{Y}^\prime(\mathbb{R}^N): \|f\|_{\dot{B}_{p,r}^s}<+\infty{\Big\}},
$$
for $s\in\mathbb{R}$, $1\leq p, r\leq +\infty$, where
$$
\|f\|_{\dot{B}_{p,r}^s}\equ\|2^{ks}\|\dot{\Delta}_k f\|_{L^p}\|_{\ell^r}.
$$

Next, we introduce the so-called Chemin-Lerner space $\widetilde{L}_T^\rho(\dot{B}_{p,r}^s)$ (see\,\cite{chemin1995flot}):
$$
\widetilde{L}_T^\rho(\dot{B}_{p,r}^s)={\Big\{}f\in (0,+\infty)\times\mathcal{Y}^\prime(\mathbb{R}^N):
\|f\|_{\widetilde{L}_T^\rho(\dot{B}_{p,r}^s)}<+\infty{\Big\}},
$$
where
$$
\|f\|_{\widetilde{L}_T^\rho(\dot{B}_{p,r}^s)}\equ\bigl{\|}2^{ks}\|\dot{\Delta}_k f(t)\|_{L^\rho(0,T;L^p)}\bigr{\|}_{\ell^r}.
$$
The index $T$ will be omitted if $T=+\infty$. A direct application of Minkowski's inequality implies that
$$
L_T^\rho(\dot{B}_{p,r}^s)\hookrightarrow \widetilde{L}_T^\rho(\dot{B}_{p,r}^s)\,\,\,{\rm if}\,\,\,r\geq \rho,
\quad\,{\rm and}\quad\,
\widetilde{L}_T^\rho(\dot{B}_{p,r}^s)\hookrightarrow {L}_T^\rho(\dot{B}_{p,r}^s)\,\,\,{\rm if}\,\,\,\rho\geq r.
$$

We shall denote by $\widetilde{\mathcal{C}}_b([0,T]; \dot{B}^s_{p,r})$ the subset of functions of $\widetilde{L}^\infty_T(\dot{B}^s_{p,r})$ which are also continuous from
$[0,T]$ to $\dot{B}^s_{p,r}$.
Also, for a tempered distribution $f$ and a universal integer $k_0$, we denote
$$
f^\ell\equ\sum_{k\leq k_0}\dot{\Delta}_kf,\,\,\,\,\,\,\,f^h\equ f-f^\ell.
$$

We will repeatedly use the following Bernstein's inequality throughout the paper:
\begin{lem}\label{le2.1}
{\rm(}see {\rm \cite{chemin1998perfect}}{\rm )} Let $\mathcal{C}$ be an annulus and $\mathcal{B}$ a ball, $1\leq p\leq q\leq +\infty$. Assume that $f\in L^p(\mathbb{R}^N)$, then for any nonnegative integer $k$, there exists constant $C$ independent of $f$, $k$ such that
$$
{\rm supp} \hat{f}\subset\lambda \mathcal{B}\Rightarrow\|D^k f\|_{L^q(\mathbb{R}^N)}:=\sup_{|\alpha|=k}\|\partial^\alpha f\|_{L^q(\mathbb{R}^N)}\leq C^{k+1}\lambda^{k+N(\frac{1}{p}-\frac{1}{q})}\|f\|_{L^p(\mathbb{R}^N)},
$$
\be\nonumber
{\rm supp} \hat{f}\subset\lambda\mathcal{C}\Rightarrow C^{-k-1}\lambda^k\|f\|_{L^p(\mathbb{R}^N)}\leq \|D^k f\|_{L^p(\mathbb{R}^N)}\leq
C^{k+1}\lambda^k\|f\|_{L^p(\mathbb{R}^N)}.
\ee
\end{lem}

Let us now state some classical properties for the Besov spaces.
\begin{prop}\label{pr2.1} The following properties hold true:

\medskip
{\rm 1)} Derivation: There exists a universal constant $C$ such that
$$
C^{-1}\|f\|_{\dot{B}_{p,r}^s}\leq \|\nabla f\|_{\dot{B}_{p,r}^{s-1}}\leq C\|f\|_{\dot{B}_{p,r}^s}.
$$

{\rm 2)} Sobolev embedding: If $1\leq p_1\leq p_2\leq\infty$ and $1\leq r_1\leq r_2\leq\infty$, then $\dot{B}_{p_1, r_1}^s\hookrightarrow \dot{B}_{p_2, r_2}^{s-\frac{N}{p_1}+\frac{N}{p_2}}$.

{\rm 3)} Real interpolation: $\|f\|_{\dot{B}_{p,r}^{\theta s_1+(1-\theta)s_2}}\leq \|f\|_{\dot{B}_{p,r}^{s_1}}^{\theta}\|f\|_{\dot{B}_{p,r}^{s_2}}^{1-\theta}$.

{\rm 4)} Algebraic properties: for $s>0$, $\dot{B}_{p,1}^s\cap L^\infty$ is an algebra.
\end{prop}
\subsection{Product estimates} We recall a few nonlinear estimates in Besov spaces which may be derived
by using paradifferential calculus.  Introduced by  Bony in \cite{bony1981calcul}, the paraproduct between $f$
and $g$ is defined by
$$
T_f g=\sum_{k\in\mathbb{Z}}\dot{S}_{k-1}f\dot{\Delta}_k g,
$$
and the remainder is given by
$$
R(f,g)=\sum_{k\in\mathbb{Z}}\dot{\Delta}_k f\widetilde{\dot{\Delta}}_k g\,\,\,\,\,\,{\rm with}\,\,\,\,\,\,\widetilde{\dot{\Delta}}_k g\equ(\dot{\Delta}_{k-1}+\dot{\Delta}_{k}+\dot{\Delta}_{k+1})g.
$$
One has  the following so-called Bony's decomposition:
\be\label{2.3}
fg=T_g f+T_f g+R(f,g).
\ee

The paraproduct $T$ and the remainder $R$ operators satisfy the following continuous properties (see e.g. \cite{bahouri2011fourier}).
 \begin{prop}\label{pr2.2}
Suppose that $s\in\mathbb{R}, \sigma>0,$ and $1\leq p, p_1, p_2, r, r_1, r_2\leq \infty$.  Then we have
 \medskip

{\rm 1)} The paraproduct $T$ is a bilinear, continuous operator from $L^\infty\times\dot{B}_{p,r}^s$ to $\dot{B}_{p,r}^s$, and from $\dot{B}_{\infty, r_1}^{-\sigma}\times\dot{B}_{p,r_2}^s$ to $\dot{B}_{p,r}^{s-\sigma}$ with
$\frac{1}{r}=\min\{1, \frac{1}{r_1}+\frac{1}{r_2}\}$.

\medskip
{\rm 2)} The remainder $R$ is bilinear continuous from $\dot{B}_{p_1,r_1}^{s_1}\times\dot{B}_{p_2,r_2}^{s_2}$ to $\dot{B}_{p,r}^{s_1+s_2}$ with $s_1+s_2>0$, $\frac{1}{p}=\frac{1}{p_1}+\frac{1}{p_2}\leq 1$, and $\frac{1}{r}=\frac{1}{r_1}+\frac{1}{r_2}\leq 1$.

 \end{prop}
 \medbreak
The following  non-classical product estimates enable us to establish the evolution of Besov norms at low frequencies
(see Lemma \ref{le3} below).
\begin{prop}\label{pr4.1}  {\rm (}\cite{bie2019optimal}{\rm )}
Assume that $1-\frac{N}{2}<\sigma_1\leq \frac{2N}{p}-\frac{N}{2} (N\geq 2)$ and $p$ satisfies \eqref{1.100}.
Then the following estimates  hold true:
\be\label{4.1}
\|fg\|_{\dot{B}_{2,\infty}^{-\sigma_1}}\lesssim\|f\|_{\dot{B}_{p,1}^{\frac{N}{p}}}\|g\|_{\dot{B}_{2,\infty}^{-\sigma_1}},
\ee
and
\be\label{4.2}
\|fg\|_{\dot{B}_{2,\infty}^{-\sigma_1}}^\ell\lesssim\|f\|_{\dot{B}_{p,1}^{\frac{N}{p}-1}}
\left(\|g\|_{\dot{B}_{p,\infty}^{-\sigma_1+\frac{N}{p}-\frac{N}{2}+1}}
+\|g\|_{\dot{B}_{p,\infty}^{-\sigma_1+\frac{2N}{p}-N+1}}\right).
\ee
\end{prop}

From Bony's decomposition \eqref{2.3} and Proposition \ref{pr2.2}, we could as well infer the following product estimates:
\begin{col}\label{co2.1} {\rm(}\cite{bahouri2011fourier}, \cite{danchin2002zero}{\rm )} If $u\in\dot{B}_{p_1,1}^{s_1}$ and $v\in\dot{B}_{p_2,1}^{s_2}$ with $1\leq p_1\leq p_2\leq \infty,~s_1\leq \frac{N}{p_1},~s_2\leq \frac{N}{p_2}$ and $s_1+s_2>0$, then
$uv\in\dot{B}_{p_2,1}^{s_1+s_2-\frac{N}{p_1}}$ and there exists a constant $C$, depending only on $N, s_1, s_2, p_1$ and $p_2$, such that
\be\nonumber
\|uv\|_{\dot{B}_{p_2,1}^{s_1+s_2-\frac{N}{p_1}}}\leq C\|u\|_{\dot{B}_{p_1,1}^{s_1}}
\|v\|_{\dot{B}_{p_2,1}^{s_2}}.
\ee
\end{col}

\begin{col}\label{co2.2}
 Assume that $1-\frac{N}{2}<\sigma_1\leq\frac{2N}{p}-\frac{N}{2}\,(N\geq 2)$ and  $p$ fulfills \eqref{1.100}, then we have
\be\nonumber
\|fg\|_{\dot{B}_{p,\infty}^{-\sigma_1+\frac{N}{p}-\frac{N}{2}+1}}\lesssim\|f\|_{\dot{B}_{p,1}^{\frac{N}{p}}}
\|g\|_{\dot{B}_{p,\infty}^{-\sigma_1+\frac{N}{p}-\frac{N}{2}+1}},
\ee
and
\be\nonumber
\|fg\|_{\dot{B}_{p,\infty}^{-\sigma_1+\frac{2N}{p}-N+1}}\lesssim\|f\|_{\dot{B}_{p,1}^{\frac{N}{p}}}
\|g\|_{\dot{B}_{p,\infty}^{-\sigma_1+\frac{2N}{p}-N+1}}.
\ee
\end{col}

We also need the following  composition lemma (see \cite{bahouri2011fourier,danchin2000global, runst1996sobolev}).
\begin{prop}\label{pra.4}

Let $F:\mathbb{R}\rightarrow \mathbb{R}$ be smooth with $F(0)=0$. For all $1\leq p,r\leq \infty$ and $s>0$, it holds that $F(u)\in \dot{B}_{p,r}^s\cap L^\infty$ for $u\in\dot{B}_{p,r}^s\cap L^\infty$, and
$$
\|F(u)\|_{\dot{B}_{p,r}^s}\leq C\|u\|_{\dot{B}_{p,r}^s}
$$
with $C$ depending only on $\|u\|_{L^\infty}$, $F^\prime$ (and higher derivatives), $s, p$ and $N$.
\end{prop}

At last,  we present the optimal regularity estimates for the heat equation (see e.g. \cite{bahouri2011fourier}).
\begin{prop}\label{pra.6}
Let $\sigma\in\mathbb{R},\,\, (p,r)\in [1,\infty]^2$ and $1\leq \rho_2\leq \rho_1\leq \infty$. Let $u$
satisfy
\be\nonumber
\left\{\begin{array}{ll}\medskip\D
\partial_t u-\mu \Delta u=f,\\ \D
u|_{t=0}=u_0.
\end{array}
\right.
\ee
Then for all $T>0$, the following a prior estimate is satisfied:
\be\nonumber
\mu^{\frac{1}{\rho_1}}\|u\|_{\widetilde{L}_T^{\rho_1}(\dot{B}_{p,r}^{\sigma+\frac{2}{\rho_1}})}
\lesssim \|u_0\|_{\dot{B}_{p,r}^\sigma}+\mu^{\frac{1}{\rho_2}-1}
\|f\|_{\widetilde{L}_T^{\rho_2}(\dot{B}_{p,r}^{\sigma-2+\frac{2}{\rho_2}})}.
\ee
\end{prop}

\section{Estimation of $L^2$-type Besov norms at low frequencies}\label{s:4}
This section  establishes  $L^2$-type Besov norms at low frequencies, which is the main ingredient in proving Theorem \ref{th2}. Firstly, we recall some properties of compressible viscoelastic flows, which have been verified in \cite{qian2010global}.
\begin{prop}\label{pr1}
  The density $\rho$ and the deformation gradient $\bF$ of system \eqref{1.1} fulfill the following equalities:
  \begin{equation}\label{4.3}
  \nabla\cdot(\rho\bF^T)={\bf 0}\,\,\,\,\,\,{\rm and}\,\,\,\,\,\,F^{lk}\partial_lF^{ij}-F^{lj}\partial_lF^{ik}=0,
  \end{equation}
if the initial data $(\rho_0, \bF_0)$ satisfies
\begin{equation}\label{4.4}
  \nabla\cdot(\rho_0\bF_0^T)={\bf 0}\,\,\,\,\,\,{\rm and}\,\,\,\,\,\,F_0^{lk}\partial_lF_0^{ij}-F_0^{lj}\partial_lF_0^{ik}=0.
  \end{equation}
\end{prop}

From Proposition \ref{pr1}, the $i$-th component of the vector ${\rm div}(\rho\bF\bF^T)$ may be written as
\begin{equation}\label{4.5}
\partial_j(\rho F^{ik}F^{jk})=\rho F^{jk}\partial_j F^{ik}+F^{ik}\partial_j(\rho F^{jk})=\rho F^{jk}\partial_jF^{ik},
\end{equation}
where we used the first equality in \eqref{4.3}.
\begin{lemma}\label{le3}
Let $p$ satisfy \eqref{1.100} and $\sigma_1\in(1-\frac{N}{2}, \frac{2N}{p}-\frac{N}{2}]$. Then the solution $(b, \bH, \bu)$ to system \eqref{2.2} satisfies
\begin{equation}\label{4.11}
\begin{split}
&\quad\|(b, \bH, \bu)(t)\|_{\dot{B}_{2,\infty}^{-\sigma_1}}^\ell\\[1ex]
&\lesssim
\|(b_0, \bH_0, \bu_0)\|_{\dot{B}_{2,\infty}^{-\sigma_1}}^\ell
+\int_0^tA_1(\tau)\|(b, \bH, \bu)(\tau)\|_{\dot{B}_{2,\infty}^{-\sigma_1}}^\ell d\tau
+\int_0^tA_2(\tau) d\tau,
\end{split}
\end{equation}
where
\begin{equation}\nonumber
\begin{split}
A_1(t)&\equ\|(b,\bH,\bu)\|_{\dot{B}_{2,1}^{\frac{N}{2}+1}}^\ell+\|(b, \bH)\|_{\dot{B}_{p,1}^{\frac{N}{p}}}^h
+\|\bu\|_{\dot{B}_{p,1}^{\frac{N}{p}+1}}^h+\|b\|_{\dot{B}_{p,1}^{\frac{N}{p}}}^2+
\|b\|_{\dot{B}_{p,1}^{\frac{N}{p}}}\|\bu\|_{\dot{B}_{p,1}^{\frac{N}{p}+1}}^h
\end{split}
\end{equation}
and
\begin{equation}\nonumber
\begin{split}
A_2(t)&\equ\Big(\|(b,\bH, \bu)\|_{\dot{B}_{p,1}^{\frac{N}{p}}}^h\Big)^2
+\|b\|_{\dot{B}_{p,1}^{\frac{N}{p}}}^2\|b\|_{\dot{B}_{p,1}^{\frac{N}{p}}}^h\\[1ex]
&\quad\quad\quad\quad\quad\quad\quad\quad\quad\quad\quad\quad+\|(b, \bH)\|_{\dot{B}_{p,1}^{\frac{N}{p}}}^h\|\bu\|_{\dot{B}_{p,1}^{\frac{N}{p}+1}}^h
+\Big(\|b\|_{\dot{B}_{p,1}^{\frac{N}{p}}}\Big)^2\|\bu\|_{\dot{B}_{p,1}^{\frac{N}{p}+1}}^h.
\end{split}
\end{equation}
\end{lemma}
\begin{proof}
As in \cite{qian2010global},
 we introduce
 \begin{equation}\label{4.6}
 \omega=\Lambda^{-1}{\rm div}{\bf u}\,\,\,\,\,\,{\rm and}\,\,\,\,\,\,\, e^{ij}=\Lambda^{-1}\partial_j u^i,
 \end{equation}
where $\Lambda^sz\equ \mathcal{F}^{-1}(|\xi|^s\mathcal{F}z),\,\,s\in\mathbb{R}$. Applying the second equality in \eqref{4.3}, one gets
\begin{equation}\label{4.7}
\Lambda^{-1}(\partial_j\partial_k H^{ik})=-\Lambda H^{ij}-\Lambda^{-1}\partial_k(H^{lj}\partial_lH^{ik}-H^{lk}\partial_lH^{ij}).
\end{equation}
Then system \eqref{2.2} becomes
\begin{equation}\label{3.1}
\left\{
\begin{array}{ll}
\partial_t {b}+\Lambda\omega=G_1,\\[1ex]
\partial_te^{ij}-\bar{\mu}\Delta e^{ij}-(\bar{\lambda}+\bar{\mu})\partial_i\partial_j\omega+\Lambda^{-1}
\partial_i\partial_jb+\Lambda H^{ij}=G_2^{ij},\\[1ex]
\partial_tH^{ij}-\Lambda e^{ij}=G_3^{ij},\\[1ex]
\omega=-\Lambda^{-2}\partial_i\partial_j e^{ij},\,\,\,u^i=-\Lambda^{-1}\partial_j e^{ij},\\[1ex]
(b, \bH, {\bf e})|_{t=0}=(b_0, \bH_0, {\bf e}_0),
\end{array}
\right.
\end{equation}
where $G_1=-b\nabla\cdot\bu-\bu\cdot\nabla b,\,\,\,G_3^{ij}=\partial_ku^iH^{kj}-\bu\cdot \nabla H^{ij}$ and
\medskip
\begin{equation}\nonumber
\begin{split}
 G_2^{ij}&=-\Lambda^{-1}\partial_j\big(\bu\cdot\nabla u^i-H^{lk}\partial_lH^{ik}+I(b)\mathcal{A}u^i+K(b)\partial_ib\big)\\[1ex]
 &\quad-\Lambda^{-1}\partial_k(H^{lj}\partial_lH^{ik}-H^{lk}\partial_lH^{ij})
 +\displaystyle\Lambda^{-1}\partial_j\Big(\frac{1}{1+b}{\rm div}\big(2\widetilde{\mu}(b)D(\bu)+\widetilde{\lambda}(b){\rm div}\bu\,{\bI}\big)\Big)^i.
\end{split}
\end{equation}
On the other hand, we need the following auxiliary equation in subsequent estimates:
\begin{equation}\label{4.8}
\partial_iH^{ij}=-\partial_jb-G_0^j,\,\,\,\,\,\,G_0^j=\partial_i(bH^{ij}),
\end{equation}
which is deduced from the first equality in \eqref{4.3}.

Utilizing the operator $\dot{\Delta}_k$ to \eqref{3.1} and denoting $n_k\equ \dot{\Delta}_kn$, one has for all $k\in \mathbb{Z}$ that
\begin{equation}\label{3.2}
\left\{
\begin{array}{ll}
\partial_t {b}_k+\Lambda\omega_k=G_{1k},\\[1ex]
\partial_te_k^{ij}-\bar{\mu}\Delta e_k^{ij}-(\bar{\lambda}+\bar{\mu})\partial_i\partial_j\omega_k+\Lambda^{-1}
\partial_i\partial_jb_k+\Lambda H_k^{ij}=G_{2k}^{ij},\\[1ex]
\partial_tH_k^{ij}-\Lambda e_k^{ij}=G_{3k}^{ij},\\[1ex]
\omega_k=-\Lambda^{-2}\partial_i\partial_je_k^{ij}.
\end{array}
\right.
\end{equation}
Taking $L^2$ scalar product of \eqref{3.2}$_2$ with $e_k^{ij}$ and thanks to \eqref{3.2}$_4$, we derive that
\begin{equation}\label{3.3}
\frac{1}{2}\frac{d}{dt}\|{\bf e}_k\|_{L^2}^2+\bar{\mu}\|\Lambda {\bf e}_k\|_{L^2}^2+(\bar{\mu}+\bar{\lambda})\|\Lambda\omega_k\|_{L^2}^2-(b_k, \Lambda\omega_k)+(\Lambda \bH_k, {\bf e}_k)=({\bf G}_{2k}, {\bf e}_k).
\end{equation}

Taking $L^2$ inner product of \eqref{3.2}$_1$ and \eqref{3.2}$_3$ with $b_k$ and $\bH_k$, respectively, and then adding the resulting equations to \eqref{3.3}, we have
\begin{equation}\label{3.4}
\begin{split}
 &\quad\frac{1}{2}\frac{d}{dt}\big(\|b_k\|_{L^2}^2+\|\bH_k\|_{L^2}^2+\|{\bf e}_k\|_{L^2}^2\big)+\bar{\mu}\|\Lambda {\bf e}_k\|_{L^2}^2+(\bar{\mu}+\bar{\lambda})\|\Lambda\omega_k\|_{L^2}^2\\[1ex]
 &=(G_{1k}, b_k)+({\bf G}_{2k}, {\bf e}_k)+({\bf G}_{3k}, \bH_k).
\end{split}
\end{equation}

To derive the dissipation arising from $(b, \bH)$, we execute the operator $\Lambda$ to \eqref{3.2}$_1$ and take the $L^2$ inner product of the resulting equation with $-\omega_k$. Also, we take the $L^2$ inner product of \eqref{3.2}$_2$ with $\Lambda^{-1}\partial_j\partial_jb_k$. Adding those resulting equations together yields
\begin{equation}\label{3.5}
\begin{split}
&\quad-\frac{d}{dt}(\Lambda b_k, \omega_k)+\|\Lambda b_k\|_{L^2}^2-\|\Lambda\omega_k\|_{L^2}^2-(\Lambda^2\omega_k, \Lambda b_k)+(H_k^{ij}, \partial_i\partial_jb_k)\\[1ex]
&=-(\Lambda G_{1k}, \omega_k)+(G_{2k}^{ij}, \Lambda^{-1}\partial_i\partial_jb_k).
\end{split}
\end{equation}
In a similar manner, we apply $\Lambda$ to \eqref{3.2}$_3$ and then take the $L^2$ inner product of the resulting equation with $e_k^{ij}$ and  also take the $L^2$ inner product of \eqref{3.2}$_2$ with $\Lambda H_k^{ij}$.
Then summing up them implies
\begin{equation}\label{3.1000}
\begin{split}
&\frac{d}{dt}(\Lambda \bH_k, {\bf e}_k)+\|\Lambda \bH_k\|_{L^2}^2-\|\Lambda{\bf e}_k\|_{L^2}^2-(\bar{\mu}+\bar{\lambda})(\Lambda H_k^{ij}, \partial_i\partial_j\omega_k)\\[1ex]
&\quad\quad\quad\quad\quad\quad\quad
+\bar{\mu}(\Lambda^2{\bf e}_k, \Lambda \bH_k)+(\partial_i\partial_jb_k, H_k^{ij})=(\Lambda{\bf G}_{3k}, {\bf e}_k)+({\bf G}_{2k}, \Lambda \bH_k).
\end{split}
\end{equation}

Multiplying a small constant $r>0$ which is determined later to \eqref{3.5} and \eqref{3.1000}, respectively, and then adding them to \eqref{3.4}, we infer
\begin{equation}\label{3.7}
\begin{aligned}
&\quad\frac{1}{2}\frac{d}{dt}\left(\|b_k\|_{L^2}^2+\|\bH_k\|_{L^2}^2+\|{\bf e}_k\|_{L^2}^2+2r(\Lambda\bH_k, {\bf e}_k)-2r(\Lambda b_k, \omega_k)\right)\\[1ex]
&\quad+(\bar{\mu}-r)\|\Lambda{\bf e}_k\|_{L^2}^2+(\bar{\mu}+\bar{\lambda}-r)\|\Lambda\omega_k\|_{L^2}^2
+r\|\Lambda b_k\|_{L^2}^2+r\|\Lambda\bH_k\|_{L^2}^2\\[1ex]
&\quad+r\bar{\mu}(\Lambda^2{\bf e}_k, \Lambda\bH_k)-r(\bar{\mu}+\bar{\lambda})(\Lambda H_k^{ij}, \partial_i\partial_j\omega_k)-r(\Lambda^2\omega_k, \Lambda b_k)+2r(\partial_i\partial_jb_k, H_k^{ij})
\\[1ex]
&=(G_{1k}, b_k)+({\bf G}_{2k}, {\bf e}_k)+({\bf G}_{3k}, \bH_k)
-r(\Lambda G_{1k}, \omega_k)+r(G_{2k}^{ij}, \Lambda^{-1}\partial_i\partial_jb_k)\\[1ex]
&\quad+r(\Lambda{\bf G}_{3k}, {\bf e}_k)+r({\bf G}_{2k}, \Lambda \bH_k).
\end{aligned}
\end{equation}
It follows from \eqref{4.8} that
\begin{equation}\label{3.1001}
\begin{aligned}
  (\partial_i\partial_jb_k, H_k^{ij})=(b_k, \partial_i\partial_j H_k^{ij})=(b_k, -\Delta b_k-\partial_j
  G_{0k}^j)=\|\Lambda b_k\|_{L^2}^2-(b_k, \partial_jG_{0k}^j).
\end{aligned}
\end{equation}
Putting \eqref{3.1001} to \eqref{3.7}, we achieve that
\begin{equation}\label{3.1002}
\begin{aligned}
  &\frac{d}{dt}g_{\ell, k}^2+\tilde{g}_{\ell,k}^2=(G_{1k}, b_k)+({\bf G}_{2k}, {\bf e}_k)+({\bf G}_{3k}, \bH_k)
-r(\Lambda G_{1k}, \omega_k)\\[1ex]
&\quad\quad\quad\quad\quad\quad+r(G_{2k}^{ij}, \Lambda^{-1}\partial_i\partial_jb_k)+r(\Lambda{\bf G}_{3k}, {\bf e}_k)+r({\bf G}_{2k}, \Lambda \bH_k)+2r(b_k, \partial_jG_{0k}^j),
\end{aligned}
\end{equation}
where
\begin{equation}\label{3.1003}
\begin{aligned}
\left.
\begin{array}{ll}
g_{\ell, k}^2\equ\|b_k\|_{L^2}^2+\|\bH_k\|_{L^2}^2+\|{\bf e}_k\|_{L^2}^2+2r(\Lambda\bH_k, {\bf e}_k)-2r(\Lambda b_k, \omega_k),\\[2ex]
\tilde{g}_{\ell, k}^2\equ(\bar{\mu}-r)\|\Lambda{\bf e}_k\|_{L^2}^2+(\bar{\mu}+\bar{\lambda}-r)\|\Lambda\omega_k\|_{L^2}^2
+3r\|\Lambda b_k\|_{L^2}^2+r\|\Lambda\bH_k\|_{L^2}^2\\[1ex]
\quad\quad\quad+r\bar{\mu}(\Lambda^2{\bf e}_k, \Lambda\bH_k)-r(\bar{\mu}+\bar{\lambda})(\Lambda H_k^{ij}, \partial_i\partial_j\omega_k)-r(\Lambda^2\omega_k, \Lambda b_k).
\end{array}
\right.
\end{aligned}
\end{equation}
For any fixed $k_0$, we may choose $r\approx r(\bar{\lambda}, \bar{\mu}, k_0)$ sufficiently small such that for $k\leq k_0$,
\begin{equation}\label{3.1004}
\left.
\begin{array}{ll}
g_{\ell, k}^2\approx\|b_k\|_{L^2}^2+\|\bH_k\|_{L^2}^2+\|{\bf e}_k\|_{L^2}^2,\\[1ex]
\tilde{g}_{\ell, k}^2\approx2^{2k}\big(\|b_k\|_{L^2}^2+\|\bH_k\|_{L^2}^2+\|{\bf e}_k\|_{L^2}^2\big).
\end{array}
\right.
\end{equation}
Using Cauchy-Schwarz inequality to deal with terms in the  right hand of \eqref{3.1002}, under the condition that $k\leq k_0$, we could derive the following inequality:
\begin{equation}\label{3.1005}
\frac{d}{dt}g_{\ell, k}+2^{2k}g_{\ell, k}\lesssim \|{\bf G}_{0k}\|_{L^2}+\|{ G}_{1k}\|_{L^2}+\|{\bf G}_{2k}\|_{L^2}+\|{\bf G}_{3k}\|_{L^2}.
\end{equation}
Then, for any $t>0$, integrating in time from $0$ to $t$ on both sides of \eqref{3.1005}, one derives
\begin{equation}\label{3.14}
\begin{aligned}
&\quad\|(b_k, \bH_k, {\bf e}_k)\|_{L^2}\\[1ex]
&\lesssim \|(b_{0k}, \bH_{0k}, {\bf e}_{0k})\|_{L^2}+\int_0^t\big(\|{\bf G}_{0k}\|_{L^2}+\|{ G}_{1k}\|_{L^2}+\|{\bf G}_{2k}\|_{L^2}+\|{\bf G}_{3k}\|_{L^2}\big)ds.
\end{aligned}
\end{equation}
Multiplying $2^{k(-\sigma_1)}$ in \eqref{3.14} and taking supremum in terms of $k\leq k_0$, we arrive at
\begin{equation}\label{3.15}
\|(b, \bH, {\bf e})(t)\|_{\dot{B}_{2,\infty}^{-\sigma_1}}^\ell\lesssim \|(b_{0}, \bH_{0}, {\bf e}_{0})\|_{\dot{B}_{2,\infty}^{-\sigma_1}}^\ell+\int_0^t\|({\bf G}_0,{G}_1,{\bf G}_2, {\bf G}_3)\|_{\dot{B}_{2,\infty}^{-\sigma_1}}^\ell ds.
\end{equation}
which combined with the relation ${u^i}=-\Lambda^{-1}\partial_j e^{ij}$  gives that
\begin{equation}\label{3.1006}
\|(b, \bH, \bu)(t)\|_{\dot{B}_{2,\infty}^{-\sigma_1}}^\ell\lesssim \|(b_{0}, \bH_{0}, {\bf u}_{0})\|_{\dot{B}_{2,\infty}^{-\sigma_1}}^\ell+\int_0^t\|({\bf G}_0,{G}_1,{\bf G}_2, {\bf G}_3)\|_{\dot{B}_{2,\infty}^{-\sigma_1}}^\ell ds.
\end{equation}

In what follows, we focus on estimates of nonlinear norm $\|({\bf G}_0,{G}_1,{\bf G}_2, {\bf G}_3)\|_{\dot{B}_{2,\infty}^{-\sigma_1}}^\ell$. Firstly, we estimate the term $G_0^j=\partial_i(bH^{ij})=\partial_ibH^{ij}+b\partial_iH^{ij}$.

\underline{Estimate of $\partial_ibH^{ij}$ and $b\partial_iH^{ij}$}. We only deal with the term $\partial_ibH^{ij}$ and the term $b\partial_iH^{ij}$ could be handled similarly. Decompose $\partial_ibH^{ij}$ as
$$\partial_ibH^{ij}=(\partial_ib)^\ell(H^{ij})^\ell+(\partial_ib)^\ell(H^{ij})^h+
(\partial_ib)^h(H^{ij})^\ell+
(\partial_ib)^h(H^{ij})^h.$$
Due to  \eqref{4.1}, we infer that
\be\label{4.14}
\|(\partial_ib)^\ell(H^{ij})^\ell\|_{\dot{B}_{2,\infty}^{-\sigma_1}}\lesssim \|(\partial_ib)^\ell\|_{\dot{B}_{p,1}^{\frac{N}{p}}}
\|(H^{ij})^\ell\|_{\dot{B}_{2,\infty}^{-\sigma_1}}\lesssim \|b\|_{\dot{B}_{2,1}^{\frac{N}{2}+1}}^\ell
\|\bH\|_{\dot{B}_{2,\infty}^{-\sigma_1}}^\ell,
\ee
and
\be\label{4.15}
\|(\partial_ib)^\ell(H^{ij})^h\|_{\dot{B}_{2,\infty}^{-\sigma_1}}\lesssim \|(H^{ij})^h\|_{\dot{B}_{p,1}^{\frac{N}{p}}}
\|(\partial_ib)^\ell\|_{\dot{B}_{2,\infty}^{-\sigma_1}}\lesssim \|\bH\|_{\dot{B}_{p,1}^{\frac{N}{p}}}^h
\|b\|_{\dot{B}_{2,\infty}^{-\sigma_1}}^\ell.
\ee
By means of \eqref{4.2}, one gets
\be\label{4.16}
\begin{split}
\|(\partial_ib)^h(H^{ij})^\ell\|_{\dot{B}_{2,\infty}^{-\sigma_1}}^\ell&\lesssim \|(\partial_ib)^h\|_{\dot{B}_{p,1}^{\frac{N}{p}-1}}
\Big(\|(H^{ij})^\ell\|_{\dot{B}_{p,\infty}^{-\sigma_1+\frac{N}{p}-\frac{N}{2}+1}}
+\|(H^{ij})^\ell\|_{\dot{B}_{p,\infty}^{-\sigma_1+\frac{2N}{p}-N+1}}\Big)\\[1ex]
&\lesssim\|b\|_{\dot{B}_{p,1}^{\frac{N}{p}}}^h\|\bH\|_{\dot{B}_{p,\infty}^{-\sigma_1+\frac{2N}{p}-N+1}}^\ell
\lesssim\|b\|_{\dot{B}_{p,1}^{\frac{N}{p}}}^h\|\bH\|_{\dot{B}_{2,\infty}^{-\sigma_1}}^\ell,
\end{split}
\ee
where we used that $-\sigma_1+\frac{2N}{p}-N+1\leq -\sigma_1+\frac{N}{p}-\frac{N}{2}+1$ in the second inequality
and the embedding ${\dot{B}_{2,\infty}^{-\sigma_1}}
\hookrightarrow{\dot{B}_{p,\infty}^{-\sigma_1+\frac{2N}{p}-N+1}}$  at the low frequency in the last inequality when $2\leq p\leq \frac{2N}{N-2}$. For the term $(\partial_ib)^h(H^{ij})^h$,  by \eqref{4.2} again, we have
\be\label{4.17}
\begin{split}
\|(\partial_ib)^h(H^{ij})^h\|_{\dot{B}_{2,\infty}^{-\sigma_1}}^\ell&\lesssim \|(\partial_ib)^h\|_{\dot{B}_{p,1}^{\frac{N}{p}-1}}
\Big(\|(H^{ij})^h\|_{\dot{B}_{p,\infty}^{-\sigma_1+\frac{N}{p}-\frac{N}{2}+1}}
+\|(H^{ij})^h\|_{\dot{B}_{p,\infty}^{-\sigma_1+\frac{2N}{p}-N+1}}\Big)\\[1ex]
&\lesssim\|b\|_{\dot{B}_{p,1}^{\frac{N}{p}}}^h
\|\bH\|_{\dot{B}_{p,\infty}^{-\sigma_1+\frac{N}{p}-\frac{N}{2}+1}}^h
\lesssim\|b\|_{\dot{B}_{p,1}^{\frac{N}{p}}}^h\|\bH\|_{\dot{B}_{p,1}^{\frac{N}{p}}}^h.
\end{split}
\ee
where we used that $-\sigma_1+\frac{2N}{p}-N+1\leq -\sigma_1+\frac{N}{p}-\frac{N}{2}+1$ in the second inequality when $p\geq 2$,
and the embedding $\dot{B}_{p,1}^{\frac{N}{p}}
\hookrightarrow{\dot{B}_{p,\infty}^{-\sigma_1+\frac{N}{p}-\frac{N}{2}+1}}$  at the high frequency in the last inequality when $\sigma_1>1-\frac{N}{2}$.

Next, we deal with $G_1=-b\,{\rm div}\bu-\bu\cdot\nabla b$ and  $G_3^{ij}=\partial_ku^iH^{kj}-\bu\cdot \nabla H^{ij}$. We only estimate $G_1$, since the two terms in $G_3^{ij}$ could be treated similarly.

\underline{Estimate of $b\,{\rm div}\bu$}. We decompose
$$b\,{\rm div}\bu=b^\ell{\rm div}\bu+b^h{\rm div}\bu^\ell+b^h{\rm div}\bu^h.$$
Thanks to \eqref{4.1}, we get
\be\label{4.18}
\|b^\ell{\rm div}\bu\|_{\dot{B}_{2,\infty}^{-\sigma_1}}\lesssim \|{\rm div}\bu\|_{\dot{B}_{p,1}^{\frac{N}{p}}}
\|b\|_{\dot{B}_{2,\infty}^{-\sigma_1}}^\ell\lesssim \Big(\|\bu\|_{\dot{B}_{2,1}^{\frac{N}{2}+1}}^\ell+\|\bu\|_{\dot{B}_{p,1}^{\frac{N}{p}+1}}^h\Big)
\|b\|_{\dot{B}_{2,\infty}^{-\sigma_1}}^\ell
\ee
and
\be\label{4.19}
\|b^h{\rm div}\bu^\ell\|_{\dot{B}_{2,\infty}^{-\sigma_1}}\lesssim \|b^h\|_{\dot{B}_{p,1}^{\frac{N}{p}}}
\|{\rm div}\bu^\ell\|_{\dot{B}_{2,\infty}^{-\sigma_1}}\lesssim \|b\|_{\dot{B}_{p,1}^{\frac{N}{p}}}^h
\|\bu\|_{\dot{B}_{2,\infty}^{-\sigma_1}}^\ell.
\ee
In view of   \eqref{4.2}, one derives
\be\label{4.20}
\begin{split}
\|b^h{\rm div}\bu^h\|_{\dot{B}_{2,\infty}^{-\sigma_1}}^\ell&\lesssim \|b^h\|_{\dot{B}_{p,1}^{\frac{N}{p}-1}}
\Big(\|{\rm div}\bu^h\|_{\dot{B}_{p,\infty}^{-\sigma_1+\frac{N}{p}-\frac{N}{2}+1}}+\|{\rm div}\bu^h\|_{\dot{B}_{p,\infty}^{-\sigma_1+\frac{2N}{p}-N+1}}\Big)\\[1ex]
&\lesssim\|b\|_{\dot{B}_{p,1}^{\frac{N}{p}}}^h\|\bu\|_{\dot{B}_{p,1}^{\frac{N}{p}+1}}^h,
\end{split}
\ee
where we used that
$
-\sigma_1+\frac{2N}{p}-N+2\leq -\sigma_1+\frac{N}{p}-\frac{N}{2}+2<\frac{N}{p}+1
$
since $\sigma_1>1-\frac{N}{2}$ and $p\geq 2$.

\underline{Estimate of $\bu\cdot\nabla b$}. Decomposing $\bu\cdot\nabla b=\bu^\ell\cdot\nabla b^\ell+\bu^h\cdot\nabla b^\ell+\bu^\ell\cdot\nabla b^h+\bu^h\cdot\nabla b^h$, we deduce from \eqref{4.1} that
\be\label{4.21}
\|\bu^\ell\nabla b^\ell\|_{\dot{B}_{2,\infty}^{-\sigma_1}}\lesssim\|\nabla b^\ell\|_{\dot{B}_{p,1}^{\frac{N}{p}}}
\|\bu^\ell\|_{\dot{B}_{2,\infty}^{-\sigma_1}}\lesssim\|b\|_{\dot{B}_{2,1}^{\frac{N}{2}+1}}^\ell
\|\bu\|_{\dot{B}_{2,\infty}^{-\sigma_1}}^\ell,
\ee
and
\be\label{4.22}
\|\bu^h\nabla b^\ell\|_{\dot{B}_{2,\infty}^{-\sigma_1}}\lesssim\|\bu^h\|_{\dot{B}_{p,1}^{\frac{N}{p}}}
\|\nabla b^\ell\|_{\dot{B}_{2,\infty}^{-\sigma_1}}\lesssim\|\bu\|_{\dot{B}_{p,1}^{\frac{N}{p}+1}}^h
\|b\|_{\dot{B}_{2,\infty}^{-\sigma_1}}^\ell.
\ee
Similar to \eqref{4.16}, one arrives at
\be\label{4.23}
\begin{split}
\|\bu^\ell\nabla b^h\|_{\dot{B}_{2,\infty}^{-\sigma_1}}^\ell&\lesssim\|\nabla b^h\|_{\dot{B}_{p,1}^{\frac{N}{p}-1}}
\Big(\|\bu^\ell\|_{\dot{B}_{p,\infty}^{-\sigma_1+\frac{N}{p}-\frac{N}{2}+1}}
+\|\bu^\ell\|_{\dot{B}_{p,\infty}^{-\sigma_1+\frac{2N}{p}-N+1}}\Big)\\[1ex]
&\lesssim\|b^h\|_{\dot{B}_{p,1}^{\frac{N}{p}}}\|\bu^\ell\|_{\dot{B}_{p,\infty}^{-\sigma_1+\frac{2N}{p}-N+1}}
\lesssim\|b\|_{\dot{B}_{p,1}^{\frac{N}{p}}}^h\|\bu\|_{\dot{B}_{2,\infty}^{-\sigma_1}}^\ell.
\end{split}
\ee
For the term $\bu^h\nabla b^h$,  by \eqref{4.2} again, one infers that
\be\label{4.24}
\begin{split}
\|\bu^h\nabla b^h\|_{\dot{B}_{2,\infty}^{-\sigma_1}}^\ell&\lesssim\|\nabla b^h\|_{\dot{B}_{p,1}^{\frac{N}{p}-1}}
\Big(\|\bu^h\|_{\dot{B}_{p,\infty}^{-\sigma_1+\frac{N}{p}-\frac{N}{2}+1}}
+\|\bu^h\|_{\dot{B}_{p,\infty}^{-\sigma_1+\frac{2N}{p}-N+1}}\Big)\\[1ex]
&\lesssim\|b^h\|_{\dot{B}_{p,1}^{\frac{N}{p}}}\|\bu^h\|_{\dot{B}_{p,\infty}^{-\sigma_1+\frac{N}{p}-\frac{N}{2}+1}}
\lesssim\|b\|_{\dot{B}_{p,1}^{\frac{N}{p}}}^h\|\bu\|_{\dot{B}_{p,1}^{\frac{N}{p}+1}}^h,
\end{split}
\ee
where we have applied the fact that
$
-\sigma_1+\frac{2N}{p}-N+1\leq -\sigma_1+\frac{N}{p}-\frac{N}{2}+1\leq \frac{N}{p}+1,
$
since $\sigma_1> 1-\frac{N}{2}$ and $p\geq 2$.

In what follows, we handle the term ${\bf G}_2$ which is expressed as
\begin{equation}\nonumber
\begin{split}
 G_2^{ij}&=-\Lambda^{-1}\partial_j\big(\bu\cdot\nabla u^i-H^{lk}\partial_lH^{ik}+I(b)\mathcal{A}u^i+K(b)\partial_ib\big)\\[1ex]
 &\quad-\Lambda^{-1}\partial_k(H^{lj}\partial_lH^{ik}-H^{lk}\partial_lH^{ij})
 +\displaystyle\Lambda^{-1}\partial_j\Big(\frac{1}{1+b}{\rm div}\big(2\widetilde{\mu}(b)D(\bu)+\widetilde{\lambda}(b){\rm div}\bu\,{\bI}\big)\Big)^i.
\end{split}
\end{equation}
Obviously, the operators $\Lambda^{-1}\partial_j$ and $\Lambda^{-1}\partial_k$  are homogeneous of degree zero.

\underline{Estimate of $\bu\cdot\nabla\bu$}. Decompose $\bu\cdot\nabla\bu=\bu^\ell\cdot\nabla\bu^\ell+\bu^\ell\cdot\nabla\bu^h
+\bu^h\cdot\nabla\bu^\ell+\bu^h\cdot\nabla\bu^h$.
It holds from \eqref{4.1} that
\be\label{4.25}
\|\bu^\ell\cdot\nabla \bu^\ell\|_{\dot{B}_{2,\infty}^{-\sigma_1}}\lesssim\|\nabla \bu^\ell\|_{\dot{B}_{p,1}^{\frac{N}{p}}}
\|\bu^\ell\|_{\dot{B}_{2,\infty}^{-\sigma_1}}\lesssim\|\bu\|_{\dot{B}_{2,1}^{\frac{N}{2}+1}}^\ell
\|\bu\|_{\dot{B}_{2,\infty}^{-\sigma_1}}^\ell,
\ee
and
\be\label{4.26}
\|\bu^h\cdot\nabla \bu^\ell\|_{\dot{B}_{2,\infty}^{-\sigma_1}}\lesssim\|\bu^h\|_{\dot{B}_{p,1}^{\frac{N}{p}}}
\|\nabla\bu^\ell\|_{\dot{B}_{2,\infty}^{-\sigma_1}}\lesssim\|\bu\|_{\dot{B}_{p,1}^{\frac{N}{p}+1}}^h
\|\bu\|_{\dot{B}_{2,\infty}^{-\sigma_1}}^\ell.
\ee
Using \eqref{4.2} yields
\be\label{4.124}
\begin{split}
\|\bu^\ell\cdot\nabla \bu^h\|_{\dot{B}_{2,\infty}^{-\sigma_1}}^\ell&\lesssim\|\nabla \bu^h\|_{\dot{B}_{p,1}^{\frac{N}{p}-1}}
\Big(\|\bu^\ell\|_{\dot{B}_{p,\infty}^{-\sigma_1+\frac{N}{p}-\frac{N}{2}+1}}
+\|\bu^\ell\|_{\dot{B}_{p,\infty}^{-\sigma_1+\frac{2N}{p}-N+1}}\Big)\\[1ex]
&\lesssim\|\bu^h\|_{\dot{B}_{p,1}^{\frac{N}{p}}}\|\bu^\ell\|_{\dot{B}_{p,\infty}^{-\sigma_1+\frac{2N}{p}-N+1}}
\lesssim\|\bu\|_{\dot{B}_{p,1}^{\frac{N}{p}+1}}^h\|\bu\|_{\dot{B}_{2,\infty}^{-\sigma_1}}^\ell,
\end{split}
\ee
and
\be\label{4.27}
\begin{split}
\|\bu^h\cdot\nabla \bu^h\|_{\dot{B}_{2,\infty}^{-\sigma_1}}^\ell&\lesssim\|\nabla \bu^h\|_{\dot{B}_{p,1}^{\frac{N}{p}-1}}
\Big(\|\bu^h\|_{\dot{B}_{p,\infty}^{-\sigma_1+\frac{N}{p}-\frac{N}{2}+1}}
+\|\bu^h\|_{\dot{B}_{p,\infty}^{-\sigma_1+\frac{2N}{p}-N+1}}\Big)\\[1ex]
&\lesssim\|\bu^h\|_{\dot{B}_{p,1}^{\frac{N}{p}}}\|\bu^h\|_{\dot{B}_{p,\infty}^{-\sigma_1+\frac{N}{p}-\frac{N}{2}+1}}
\lesssim\|\bu\|_{\dot{B}_{p,1}^{\frac{N}{p}}}^h\|\bu\|_{\dot{B}_{p,1}^{\frac{N}{p}}}^h.
\end{split}
\ee

\underline{Estimate of $H^{lj}\partial_lH^{ik}, H^{lk}\partial_lH^{ij}$ and $H^{lk}\partial_lH^{ik}$}. These three terms could be handled similar to the term $\partial_ibH^{ij}$ and  here we omit the details for simplicity.

\underline{Estimate of $I(b)\mathcal{A}\bu$}. Keeping in mind that $I(0)=0$, one may write
$$
I(b)=I^\prime(0)b+\bar{I}(b)b
$$
for some smooth function $\bar{I}$ vanishing at $0$. Thus, using \eqref{4.1} again, we have
\be\label{4.28}
\|b^\ell\mathcal{A}\bu^\ell\|_{\dot{B}_{2,\infty}^{-\sigma_1}}\lesssim\|\mathcal{A}\bu^\ell\|_{\dot{B}_{p,1}^{\frac{N}{p}}}
\|b^\ell\|_{\dot{B}_{2,\infty}^{-\sigma_1}}\lesssim\|\bu\|_{\dot{B}_{2,1}^{\frac{N}{2}+1}}^\ell
\|b\|_{\dot{B}_{2,\infty}^{-\sigma_1}}^\ell,
\ee
and
\be\label{4.29}
\|b^h\mathcal{A}\bu^\ell\|_{\dot{B}_{2,\infty}^{-\sigma_1}}\lesssim\|b^h\|_{\dot{B}_{p,1}^{\frac{N}{p}}}
\|\mathcal{A}\bu^\ell\|_{\dot{B}_{2,\infty}^{-\sigma_1}}\lesssim\|b\|_{\dot{B}_{p,1}^{\frac{N}{p}}}^h
\|\bu\|_{\dot{B}_{2,\infty}^{-\sigma_1}}^\ell.
\ee
Arguing similarly as deriving \eqref{4.124} and \eqref{4.27}, one has
\be\label{4.30}
\begin{split}
\|b^\ell\mathcal{A} \bu^h\|_{\dot{B}_{2,\infty}^{-\sigma_1}}^\ell&\lesssim\|\mathcal{A} \bu^h\|_{\dot{B}_{p,1}^{\frac{N}{p}-1}}
\Big(\|b^\ell\|_{\dot{B}_{p,\infty}^{-\sigma_1+\frac{N}{p}-\frac{N}{2}+1}}
+\|b^\ell\|_{\dot{B}_{p,\infty}^{-\sigma_1+\frac{2N}{p}-N+1}}\Big)\\[1ex]
&\lesssim\|\bu\|_{\dot{B}_{p,1}^{\frac{N}{p}+1}}^h\|b^\ell\|_{\dot{B}_{p,\infty}^{-\sigma_1+\frac{2N}{p}-N+1}}
\lesssim\|\bu\|_{\dot{B}_{p,1}^{\frac{N}{p}+1}}^h\|b\|_{\dot{B}_{2,\infty}^{-\sigma_1}}^\ell,
\end{split}
\ee
and
\be\label{4.31}
\begin{split}
\|b^h\mathcal{A}\bu^h\|_{\dot{B}_{2,\infty}^{-\sigma_1}}^\ell&\lesssim\|\mathcal{A}\bu^h\|_{\dot{B}_{p,1}^{\frac{N}{p}-1}}
\Big(\|b^h\|_{\dot{B}_{p,\infty}^{-\sigma_1+\frac{N}{p}-\frac{N}{2}+1}}
+\|b^h\|_{\dot{B}_{p,\infty}^{-\sigma_1+\frac{2N}{p}-N+1}}\Big)\\[1ex]
&\lesssim\|\bu^h\|_{\dot{B}_{p,1}^{\frac{N}{p}+1}}\|b^h\|_{\dot{B}_{p,\infty}^{-\sigma_1+\frac{N}{p}-\frac{N}{2}+1}}
\lesssim\|\bu\|_{\dot{B}_{p,1}^{\frac{N}{p}+1}}^h\|b\|_{\dot{B}_{p,1}^{\frac{N}{p}}}^h.
\end{split}
\ee
On the other hand, from \eqref{4.1}, \eqref{4.2}, Proposition \ref{pra.4} and Corollaries \ref{co2.1} and \ref{co2.2}, we have
\be\label{4.32}
\|\bar{I}(b)b\mathcal{A}\bu^\ell\|_{\dot{B}_{2,\infty}^{-\sigma_1}}
\lesssim\|\bar{I}(b)b\|_{\dot{B}_{p,1}^{\frac{N}{p}}}\|\mathcal{A}\bu^\ell\|_{\dot{B}_{2,\infty}^{-\sigma_1}}
\lesssim\|b\|_{\dot{B}_{p,1}^{\frac{N}{p}}}^2\|\bu\|_{\dot{B}_{2,\infty}^{-\sigma_1}}^\ell,
\ee
and
\be\label{4.33}
\begin{split}
\|\bar{I}(b)b\mathcal{A}\bu^h\|_{\dot{B}_{2,\infty}^{-\sigma_1}}^\ell
&\lesssim\|\mathcal{A}\bu^h\|_{\dot{B}_{p,1}^{\frac{N}{p}-1}}
\Big(\|\bar{I}(b)b\|_{\dot{B}_{p,\infty}^{-\sigma_1+\frac{N}{p}-\frac{N}{2}+1}}
+\|\bar{I}(b)b\|_{\dot{B}_{p,\infty}^{-\sigma_1+\frac{2N}{p}-N+1}}\Big)\\[1ex]
&\lesssim\|\bu^h\|_{\dot{B}_{p,1}^{\frac{N}{p}+1}}\|\bar{I}(b)\|_{\dot{B}_{p,1}^{\frac{N}{p}}}
\Big(\|b\|_{\dot{B}_{p,\infty}^{-\sigma_1+\frac{N}{p}-\frac{N}{2}+1}}
+\|b\|_{\dot{B}_{p,\infty}^{-\sigma_1+\frac{2N}{p}-N+1}}\Big)\\[1ex]
&\lesssim\|\bu\|_{\dot{B}_{p,1}^{\frac{N}{p}+1}}^h
\|b\|_{\dot{B}_{p,1}^{\frac{N}{p}}}\Big(\|b\|_{\dot{B}_{p,\infty}^{-\sigma_1+\frac{N}{p}-\frac{N}{2}+1}}^h
+\|b\|_{\dot{B}_{p,\infty}^{-\sigma_1+\frac{2N}{p}-N+1}}^\ell\Big)\\[1ex]
&\lesssim\|\bu\|_{\dot{B}_{p,1}^{\frac{N}{p}+1}}^h
\|b\|_{\dot{B}_{p,1}^{\frac{N}{p}}}\Big(\|b\|_{\dot{B}_{p,1}^{\frac{N}{p}}}^h
+\|b\|_{\dot{B}_{2,\infty}^{-\sigma_1}}^\ell\Big).
\end{split}
\ee

\underline{Estimate of $K(b)\nabla b$}.  In view of $K(0)=0$, we may write $K(b)=K^\prime(0)b+\bar{K}(b)b$, here $\bar{K}$ is a smooth function fulfilling $\bar{K}(0)=0$. For the term $b\nabla b$, we achieve
\be\label{4.34}
\|b^\ell\nabla b^\ell\|_{\dot{B}_{2,\infty}^{-\sigma_1}}\lesssim\|\nabla b^\ell\|_{\dot{B}_{p,1}^{\frac{N}{p}}}
\|b^\ell\|_{\dot{B}_{2,\infty}^{-\sigma_1}}\lesssim\|b\|_{\dot{B}_{2,1}^{\frac{N}{2}+1}}^\ell
\|b\|_{\dot{B}_{2,\infty}^{-\sigma_1}}^\ell,
\ee
and
\be\label{4.35}
\|b^h\nabla b^\ell\|_{\dot{B}_{2,\infty}^{-\sigma_1}}\lesssim\|b^h\|_{\dot{B}_{p,1}^{\frac{N}{p}}}
\|\nabla b^\ell\|_{\dot{B}_{2,\infty}^{-\sigma_1}}\lesssim\|b\|_{\dot{B}_{p,1}^{\frac{N}{p}}}^h
\|b\|_{\dot{B}_{2,\infty}^{-\sigma_1}}^\ell.
\ee
Also, one has
\be\label{4.36}
\begin{split}
\|b^\ell\nabla b^h\|_{\dot{B}_{2,\infty}^{-\sigma_1}}^\ell&\lesssim\|\nabla b^h\|_{\dot{B}_{p,1}^{\frac{N}{p}-1}}
\Big(\|b^\ell\|_{\dot{B}_{p,\infty}^{-\sigma_1+\frac{N}{p}-\frac{N}{2}+1}}
+\|b^\ell\|_{\dot{B}_{p,\infty}^{-\sigma_1+\frac{2N}{p}-N+1}}\Big)\\[1ex]
&\lesssim\|b\|_{\dot{B}_{p,1}^{\frac{N}{p}}}^h\|a^\ell\|_{\dot{B}_{p,\infty}^{-\sigma_1+\frac{2N}{p}-N+1}}
\lesssim\|b\|_{\dot{B}_{p,1}^{\frac{N}{p}}}^h\|b\|_{\dot{B}_{2,\infty}^{-\sigma_1}}^\ell,
\end{split}
\ee
and
\be\label{4.37}
\begin{split}
\|b^h\nabla b^h\|_{\dot{B}_{2,\infty}^{-\sigma_1}}^\ell&\lesssim\|\nabla b^h\|_{\dot{B}_{p,1}^{\frac{N}{p}-1}}
\Big(\|b^h\|_{\dot{B}_{p,\infty}^{-\sigma_1+\frac{N}{p}-\frac{N}{2}+1}}
+\|b^h\|_{\dot{B}_{p,\infty}^{-\sigma_1+\frac{2N}{p}-N+1}}\Big)\\[1ex]
&\lesssim\|b^h\|_{\dot{B}_{p,1}^{\frac{N}{p}}}\|b^h\|_{\dot{B}_{p,\infty}^{-\sigma_1+\frac{N}{p}-\frac{N}{2}+1}}
\lesssim\|b\|_{\dot{B}_{p,1}^{\frac{N}{p}}}^h\|b\|_{\dot{B}_{p,1}^{\frac{N}{p}}}^h.
\end{split}
\ee
In regard the term $\bar{K}(b)b\nabla b$, we use the decomposition $\bar{K}(b)b\nabla b=\bar{K}(b)b\nabla b^\ell+\bar{K}(b)b\nabla b^h$ and get from \eqref{4.1}-\eqref{4.2}, Corollary \ref{co2.2} and Proposition \ref{pra.4} again that
\be\label{4.38}
\begin{split}
\|\bar{K}(b)b\nabla b^\ell\|_{\dot{B}_{2,\infty}^{-\sigma_1}}\lesssim\|\bar{K}(b)b\|_{\dot{B}_{p,1}^{\frac{N}{p}}}
\|\nabla b\|_{\dot{B}_{2,\infty}^{-\sigma_1}}^\ell\lesssim\|b\|_{\dot{B}_{p,1}^{\frac{N}{p}}}^2
\|b\|_{\dot{B}_{2,\infty}^{-\sigma_1}}^\ell,
\end{split}
\ee
and
\be\label{4.39}
\begin{split}
\|\bar{K}(b)b\nabla b^h\|_{\dot{B}_{2,\infty}^{-\sigma_1}}^\ell&\lesssim\|\nabla b^h\|_{\dot{B}_{p,1}^{\frac{N}{p}-1}}
\Big(\|\bar{K}(b)b\|_{\dot{B}_{p,\infty}^{-\sigma_1+\frac{N}{p}-\frac{N}{2}+1}}
+\|\bar{K}(b)b\|_{\dot{B}_{p,\infty}^{-\sigma_1+\frac{2N}{p}-N+1}}\Big)\\[1ex]
&\lesssim\|b^h\|_{\dot{B}_{p,1}^{\frac{N}{p}}}\|\bar{K}(b)\|_{\dot{B}_{p,1}^{\frac{N}{p}}}
\Big(\|b\|_{\dot{B}_{p,\infty}^{-\sigma_1+\frac{N}{p}-\frac{N}{2}+1}}
+\|b\|_{\dot{B}_{p,\infty}^{-\sigma_1+\frac{2N}{p}-N+1}}\Big)\\[1ex]
&\lesssim\|b\|_{\dot{B}_{p,1}^{\frac{N}{p}}}^2\Big(\|b\|_{\dot{B}_{2,\infty}^{-\sigma_1}}^\ell
+\|b\|_{\dot{B}_{p,1}^{\frac{N}{p}}}^h\Big).
\end{split}
\ee

\underline{Estimate of $\frac{1}{1+b}(2\widetilde{\mu}(b){\rm div}D(\bu)+\widetilde{\lambda}(b)\nabla{\rm div}\bu)$}.
The estimate of this term could be similarly handled  as the term $I(b)\mathcal{A}\bu$ and the details are omitted here.

\underline{Estimate of $\frac{1}{1+b}(2\widetilde{\mu}^\prime(b)D(\bu)\cdot\nabla b+\widetilde{\lambda}^\prime(b){\rm div}\bu\nabla b)$}. We only deal with the term $\frac{2\widetilde{\mu}^\prime(b)}{1+b}D(\bu)\cdot\nabla b$ and the remainder term could be handled in a similar manner. Denote by $J(b)$ the smooth function fulfilling
\be\label{4.1000}
J^\prime(b)=\frac{2\widetilde{\mu}^\prime(b)}{1+b}
\,\,\,{\rm and}\,\,\, J(0)=0, \,\,\,{\rm so\,\, that}\,\,\, \nabla J(b)=\frac{2\widetilde{\mu}^\prime(b)}{1+b}\nabla b.
\ee
Decomposing $J(b)=J^\prime(0)b+\bar{J}(b)b$ implies
$\nabla J(b)=J^\prime(0)\nabla b+\nabla (\bar{J}(b)b)$. Then we have from \eqref{4.1} and \eqref{4.2} that
\be\label{4.40}
\|\nabla b^\ell D(\bu)^\ell\|_{\dot{B}_{2,\infty}^{-\sigma_1}}\lesssim\|\nabla b^\ell\|_{\dot{B}_{p,1}^{\frac{N}{p}}}
\|D(\bu)^\ell\|_{\dot{B}_{2,\infty}^{-\sigma_1}}\lesssim\|b\|_{\dot{B}_{2,1}^{\frac{N}{2}+1}}^\ell
\|\bu\|_{\dot{B}_{2,\infty}^{-\sigma_1}}^\ell,
\ee
\be\label{4.41}
\|\nabla b^\ell D(\bu)^h\|_{\dot{B}_{2,\infty}^{-\sigma_1}}\lesssim\|D(\bu)^h\|_{\dot{B}_{p,1}^{\frac{N}{p}}}
\|\nabla b^\ell\|_{\dot{B}_{2,\infty}^{-\sigma_1}}\lesssim\|\bu\|_{\dot{B}_{p,1}^{\frac{N}{p}+1}}^h
\|b\|_{\dot{B}_{2,\infty}^{-\sigma_1}}^\ell,
\ee
and
\be\label{4.42}
\begin{split}
\|\nabla b^h D(\bu)\|_{\dot{B}_{2,\infty}^{-\sigma_1}}^\ell&\lesssim\|\nabla b^h\|_{\dot{B}_{p,1}^{\frac{N}{p}-1}}
\Big(\|D(\bu)\|_{\dot{B}_{p,\infty}^{-\sigma_1+\frac{N}{p}-\frac{N}{2}+1}}
+\|D(\bu)\|_{\dot{B}_{p,\infty}^{-\sigma_1+\frac{2N}{p}-N+1}}\Big)\\[1ex]
&\lesssim\|b^h\|_{\dot{B}_{p,1}^{\frac{N}{p}}}
\Big(\|\bu\|_{\dot{B}_{p,\infty}^{-\sigma_1+\frac{N}{p}-\frac{N}{2}+2}}^h
+\|\bu\|_{\dot{B}_{p,\infty}^{-\sigma_1+\frac{2N}{p}-N+2}}^\ell\Big)\\[1ex]
&\lesssim\|b\|_{\dot{B}_{p,1}^{\frac{N}{p}}}^h\Big(\|\bu\|_{\dot{B}_{p,1}^{\frac{N}{p}+1}}^h
+\|\bu\|_{\dot{B}_{2,\infty}^{-\sigma_1}}^\ell
\Big).
\end{split}
\ee
In addition, the remaining term with $\bar{J}(a)a$ may be estimated as
\be\label{4.400}
\begin{split}
\|\nabla (\bar{J}(b)b) D(\bu)\|_{\dot{B}_{2,\infty}^{-\sigma_1}}^\ell&\lesssim\|\bar{J}(b)b\|_{\dot{B}_{p,1}^{\frac{N}{p}}}
\Big(\|D(\bu)\|_{\dot{B}_{p,\infty}^{-\sigma_1+\frac{N}{p}-\frac{N}{2}+1}}
+\|D(\bu)\|_{\dot{B}_{p,\infty}^{-\sigma_1+\frac{2N}{p}-N+1}}\Big)\\[1ex]
&\lesssim\|b\|_{\dot{B}_{p,1}^{\frac{N}{p}}}^2
\Big(\|\bu\|_{\dot{B}_{p,\infty}^{-\sigma_1+\frac{N}{p}-\frac{N}{2}+2}}^h
+\|\bu\|_{\dot{B}_{p,\infty}^{-\sigma_1+\frac{2N}{p}-N+2}}^\ell\Big)\\[1ex]
&\lesssim\|b\|_{\dot{B}_{p,1}^{\frac{N}{p}}}^2
\Big(\|\bu\|_{\dot{B}_{p,1}^{\frac{N}{p}+1}}^h+\|\bu\|_{\dot{B}_{2,\infty}^{-\sigma_1}}^\ell
\Big).
\end{split}
\ee
Plugging all  estimates above in \eqref{3.1006}, we end up with the proof of Lemma \ref{le3}.
\end{proof}

By the definition of $\mathcal{X}_p(t)$ in Theorem \ref{th1.1}, one has
\be\nonumber
\begin{split}
\|(b, \bH, \bu)\|_{L_t^2(\dot{B}_{p,1}^{\frac{N}{p}})}^\ell&\lesssim\|(b, \bH, \bu)\|_{L_t^2(\dot{B}_{2,1}^{\frac{N}{2}})}^\ell\lesssim\Big(\|(b, \bH, \bu)\|_{L_t^\infty(\dot{B}_{2,1}^{\frac{N}{2}-1})}^\ell\Big)^{\frac{1}{2}}\Big(\|(b, \bH, \bu)\|_{L_t^1(\dot{B}_{2,1}^{\frac{N}{2}+1})}^\ell\Big)^{\frac{1}{2}},
\end{split}
\ee
\be\nonumber
\|(b,\bH)\|_{L_t^2(\dot{B}_{p,1}^{\frac{N}{p}})}^h\lesssim\Big(\|(b,\bH)\|_{L_t^\infty(\dot{B}_{p,1}^{\frac{N}{p}})}^h\Big)^{\frac{1}{2}}
\Big(\|(b,\bH)\|_{L_t^1(\dot{B}_{p,1}^{\frac{N}{p}})}^h\Big)^{\frac{1}{2}},
\ee
and
\be\nonumber
\|\bu\|_{L_t^2(\dot{B}_{p,1}^{\frac{N}{p}})}^h\lesssim\Big(\|\bu\|_{L_t^\infty(\dot{B}_{p,1}^{\frac{N}{p}-1})}^h\Big)^{\frac{1}{2}}
\Big(\|\bu\|_{L_t^1(\dot{B}_{p,1}^{\frac{N}{p}+1})}^h\Big)^{\frac{1}{2}}.
\ee
On the other hand, it follows that
\be\nonumber
\|b\|_{L_t^\infty(\dot{B}_{p,1}^{\frac{N}{p}})}\lesssim\|b\|_{L_t^\infty(\dot{B}_{p,1}^{\frac{N}{p}})}^\ell
+\|b\|_{L_t^\infty(\dot{B}_{p,1}^{\frac{N}{p}})}^h\lesssim\|b\|_{L_t^\infty(\dot{B}_{2,1}^{\frac{N}{2}-1})}^\ell
+\|b\|_{L_t^\infty(\dot{B}_{p,1}^{\frac{N}{p}})}^h.
\ee
Then, we have
\be\label{4.73}
\int_0^t(A_1(\tau)+A_2(\tau))d\tau\leq \mathcal{X}_p+\mathcal{X}_p^2+\mathcal{X}_p^3\leq C\mathcal{X}_{p,0},
\ee
which yields from Gronwall's inequality that
\be\label{4.74}
\|(b,\bH,\bu)\|_{\dot{B}_{2,\infty}^{-\sigma_1}}^\ell\leq C_0
\ee
for all $t\geq 0$, where $C_0>0$ depends on $\|(b_0, \bH_0, \bu_0)\|_{\dot{B}_{2,\infty}^{-\sigma_1}}^\ell$ and $\mathcal{X}_{p,0}$.
\section{Proofs of main results}\label{s:5}
This section is devoted to proving Theorem \ref{th2} and Corollary \ref{col1}.
\subsection{Proof of Theorem \ref{th2}}
From Theorem \ref{th1.1}, we have
\begin{equation}\label{5.1}
\begin{split}
&\quad\|(b,\bH,\bu)(t)\|_{\dot{B}_{2,1}^{\frac{N}{2}-1}}^\ell+\|b(t)\|_{\dot{B}_{p,1}^{\frac{N}{p}}}^h
+\|\bH(t)\|_{\dot{B}_{p,1}^{\frac{N}{p}}}^h+\|\bu(t)\|_{\dot{B}_{p,1}^{\frac{N}{p}-1}}^h\\[1ex]
&+\int_0^t\Big(\|(b,\bH,\bu)(\tau)\|_{\dot{B}_{2,1}^{\frac{N}{2}+1}}^\ell
+\|b(\tau)\|_{\dot{B}_{p,1}^{\frac{N}{p}}}^h+\|\bH(\tau)\|_{\dot{B}_{p,1}^{\frac{N}{p}}}^h
+\|\bu(\tau)\|_{\dot{B}_{p,1}^{\frac{N}{p}+1}}^h\Big)d\tau\lesssim\mathcal{X}_{p,0}.
\end{split}
\end{equation}

In what follows, we will employ  the following interpolation inequalities:
\begin{prop}\label{pra.7}(\cite{xin2018optimal})
Suppose that $m\neq \rho$. Then it holds that
$$
\|f\|_{\dot{B}_{p,1}^j}^\ell\lesssim (\|f\|_{\dot{B}_{r,\infty}^m}^\ell)^{1-\eta}(\|f\|_{\dot{B}_{r,\infty}^{\rho}}^\ell)^{\eta}
\,\,\,{\rm and}\,\,\,
\|f\|_{\dot{B}_{p,1}^j}^h\lesssim (\|f\|_{\dot{B}_{r,\infty}^m}^h)^{1-\eta}(\|f\|_{\dot{B}_{r,\infty}^{\rho}}^h)^{\eta},
$$
where $j+N(\frac{1}{r}-\frac{1}{p})=m(1-\eta)+\rho\eta$ for $0<\eta<1$ and $1\leq r\leq p\leq \infty$.
\end{prop}

Due to $-\sigma_1<\frac{N}{2}-1\leq \frac{N}{p}<\frac{N}{2}+1$, it follows from Proposition \ref{pra.7} that
\be\label{5.40}
\|(b,\bH,\bu)\|_{\dot{B}_{2,1}^{\frac{N}{2}-1}}^\ell
\leq C\Big(\|(b,\bH,\bu)\|_{\dot{B}_{2,\infty}^{-\sigma_1}}^\ell\Big)^{\eta_0}
\Big(\|(b,\bH,\bu)\|_{\dot{B}_{2,\infty}^{\frac{N}{2}+1}}^\ell\Big)^{1-\eta_0},
\ee
where $\eta_0=\frac{2}{N/2+1+\sigma_1}\in (0,1)$. In view of \eqref{4.74}, we have
$$
\|(b,\bH,\bu)\|_{\dot{B}_{2,\infty}^{\frac{N}{2}+1}}^\ell\geq c_0\Big(\|(b,\bH,\bu)\|_{\dot{B}_{2,1}^{\frac{N}{2}-1}}^\ell\Big)^{\frac{1}{1-\eta_0}},
$$
where $c_0=C^{-\frac{1}{1-\eta_0}}C_0^{-\frac{\eta_0}{1-\eta_0}}$.

Moreover, it follows from the fact
$\|b\|_{\dot{B}_{p,1}^{\frac{N}{p}}}^h+\|\bH\|_{\dot{B}_{p,1}^{\frac{N}{p}}}^h
+\|\bu\|_{\dot{B}_{p,1}^{\frac{N}{p}-1}}^h\leq \mathcal{X}_p(t)\lesssim \mathcal{X}_{p,0}\ll 1$
for all $t\geq 0$ that
$$
\Big(\|b\|_{\dot{B}_{p,1}^{\frac{N}{p}}}^h\Big)^{\frac{1}{1-\eta_0}}\lesssim\|b\|_{\dot{B}_{p,1}^{\frac{N}{p}}}^h,\,\,\,\,\,
\Big(\|\bH\|_{\dot{B}_{p,1}^{\frac{N}{p}}}^h\Big)^{\frac{1}{1-\eta_0}}\lesssim
\|\bH\|_{\dot{B}_{p,1}^{\frac{N}{p}}}^h\,\,\,\,{\rm and}\,\,\,\Big(\|\bu\|_{\dot{B}_{p,1}^{\frac{N}{p}-1}}^h\Big)^{\frac{1}{1-\eta_0}}\lesssim
\|\bu\|_{\dot{B}_{p,1}^{\frac{N}{p}+1}}^h.
$$
Thus, we have the following Lyapunov-type inequality:
{\small\begin{equation}\label{5.41}
\begin{split}
&\quad\|(b,\bH,\bu)(t)\|_{\dot{B}_{2,1}^{\frac{N}{2}-1}}^\ell+\|b(t)\|_{\dot{B}_{p,1}^{\frac{N}{p}}}^h
+\|\bH(t)\|_{\dot{B}_{p,1}^{\frac{N}{p}}}^h+\|\bu(t)\|_{\dot{B}_{p,1}^{\frac{N}{p}-1}}^h\\[1ex]
&+\int_0^t\Big(\|(b,\bH,\bu)(t)\|_{\dot{B}_{2,1}^{\frac{N}{2}-1}}^\ell+\|b(t)\|_{\dot{B}_{p,1}^{\frac{N}{p}}}^h
+\|\bH(t)\|_{\dot{B}_{p,1}^{\frac{N}{p}}}^h+\|\bu(t)\|_{\dot{B}_{p,1}^{\frac{N}{p}-1}}^h\Big)^{1+\frac{2}{N/2-1+\sigma_1}}d\tau\lesssim\mathcal{X}_{p,0}.
\end{split}
\end{equation}}
Solving \eqref{5.41}  yields
\be\label{5.42}
\begin{split}
&\quad\|(b,\bH,\bu)(t)\|_{\dot{B}_{2,1}^{\frac{N}{2}-1}}^\ell+\|b(t)\|_{\dot{B}_{p,1}^{\frac{N}{p}}}^h
+\|\bH(t)\|_{\dot{B}_{p,1}^{\frac{N}{p}}}^h+\|\bu(t)\|_{\dot{B}_{p,1}^{\frac{N}{p}-1}}^h\\[1ex]
&\lesssim\Big(\mathcal{X}_{p,0}^{-\frac{2}{N/2-1+\sigma_1}}
+\frac{2t}{N/2-1+\sigma_1}\Big)^{-\frac{N/2-1+\sigma_1}{2}}
\lesssim (1+t)^{-\frac{N/2-1+\sigma_1}{2}}
\end{split}
\ee
for all $t\geq 0$.
Resorting to the embedding properties in Proposition \ref{pr2.1}, we arrive at
\be\label{5.43}
\begin{split}
&\quad\|(b, \bH, \bu)(t)\|_{\dot{B}_{p,1}^{\frac{N}{p}-1}}\\[1ex]
&\lesssim\|(b, \bH, \bu)(t)\|_{\dot{B}_{2,1}^{\frac{N}{2}-1}}^\ell+\|(b, \bH)\|_{\dot{B}_{p,1}^{\frac{N}{p}}}^h+\|\bu(t)\|_{\dot{B}_{p,1}^{\frac{N}{p}-1}}^h
\lesssim (1+t)^{-\frac{N/2-1+\sigma_1}{2}}.
\end{split}
\ee
In addition, employing Proposition
\ref{pra.7} again yields  for $\sigma\in (-\sigma_1-\frac{N}{2}+\frac{N}{p}, \frac{N}{p}-1)$ that
\be\label{5.44}
\begin{split}
\|(b, \bH, \bu)(t)\|_{\dot{B}_{p,1}^{\sigma}}^\ell&\lesssim\|(b, \bH, \bu)(t)\|_{\dot{B}_{2,1}^{\sigma+N(\frac{1}{2}-\frac{1}{p})}}^\ell\\[1ex]
&\lesssim\Big(\|(b, \bH, \bu)\|_{\dot{B}_{2,\infty}^{-\sigma_1}}^\ell\Big)^{\eta_1}
\Big(\|(b, \bH, \bu)\|_{\dot{B}_{2,\infty}^{\frac{N}{2}-1}}^\ell\Big)^{1-\eta_1},
\end{split}
\ee
where
$$
\eta_1=\frac{\frac{N}{p}-1-\sigma}{\frac{N}{2}-1+\sigma_1}\in(0,1).
$$
Note that
$
\|(b, \bH, \bu)\|_{\dot{B}_{2,\infty}^{-\sigma_1}}^\ell\leq C_0
$
for all $t\geq 0$.  From \eqref{5.42} and \eqref{5.44}, we deduce that
\be\label{5.45}
\begin{split}
\|(b, \bH, \bu)(t)\|_{\dot{B}_{p,1}^{\sigma}}^\ell\lesssim
\Big[(1+t)^{-\frac{N/2-1+\sigma_1}{2}}\Big]^{1-\eta_1}
=(1+t)^{-\frac{N}{2}(\frac{1}{2}-\frac{1}{p})-\frac{\sigma+\sigma_1}{2}}
\end{split}
\ee
for all $t\geq 0$, which leads to
\be\label{5.46}
\begin{split}
\|(b, \bH, \bu)(t)\|_{\dot{B}_{p,1}^{\sigma}}\lesssim
\|(b, \bH, \bu)(t)\|_{\dot{B}_{p,1}^{\sigma}}^\ell+\|(b, \bH, \bu)(t)\|_{\dot{B}_{p,1}^{\sigma}}^h
\lesssim (1+t)^{-\frac{N}{2}(\frac{1}{2}-\frac{1}{p})-\frac{\sigma+\sigma_1}{2}}
\end{split}
\ee
provided that $\sigma\in (-\sigma_1-N(\frac{1}{2}-\frac{1}{p}), \frac{N}{p}-1)$. This together with \eqref{5.43} yields \eqref{1}.
So far, the proof of Theorem \ref{th2} is completed.
\subsection{Proof of Corollary \ref{col1}} In fact, Corollary \ref{col1} can be regarded as the direct consequence of the following interpolation inequality:
\begin{prop}\label{pra.8}(\cite{bahouri2011fourier})
The following interpolation inequality holds true:
$$
\|\Lambda^lf\|_{L^r}\lesssim \|\Lambda^mf\|_{L^q}^{1-\eta}\|\Lambda^kf\|_{L^q}^\eta,
$$
whenever $0\leq \eta\leq 1, 1\leq q\leq r\leq \infty$ and
$$
l+N\Big(\frac{1}{q}-\frac{1}{r}\Big)=m(1-\eta)+k\eta.
$$
\end{prop}

 With the aid of  Proposition \ref{pra.8}, we define $\eta_2$ by the relation
$$
m(1-\eta_2)+k\eta_2=l+N\Big(\frac{1}{p}-\frac{1}{r}\Big),
$$
where $m=\frac{N}{p}-1$ and $k=-\sigma_1-N(\frac{1}{2}-\frac{1}{p})+\varepsilon$ with $\varepsilon>0$ small enough. When $l\in\mathbb{R}$ satisfying $-\sigma_1
-\frac{N}{2}+\frac{N}{p}<l+\frac{N}{p}-\frac{N}{r}\leq \frac{N}{p}-1$, it is easy to see that $\eta_2\in [0,1)$. As a consequence, we conclude  by $\dot{B}_{p,1}^0\hookrightarrow L^p$ that
\be\nonumber
\begin{split}
&\quad\|\Lambda^l(b, \bH, \bu)\|_{L^r}\lesssim
\|\Lambda^m(b, \bH, \bu)\|_{L^p}^{1-\eta_2}\|\Lambda^k(b, \bH, \bu)\|_{L^p}^{\eta_2}\\[1ex]
&\lesssim \left[(1+t)^{-\frac{N}{2}(\frac{1}{2}-\frac{1}{p})-\frac{m+\sigma_1}{2}}\right]^{1-\eta_2}
\left[(1+t)^{-\frac{N}{2}(\frac{1}{2}-\frac{1}{p})-\frac{k+\sigma_1}{2}}\right]^{\eta_2}
=(1+t)^{-\frac{N}{2}(\frac{1}{2}-\frac{1}{r})-\frac{l+\sigma_1}{2}}.
\end{split}
\ee
Thus, we finish the proof of Corollary \ref{col1}.
\providecommand{\href}[2]{#2}
\providecommand{\arxiv}[1]{\href{http://arxiv.org/abs/#1}{arXiv:#1}}
\providecommand{\url}[1]{\texttt{#1}}
\providecommand{\urlprefix}{URL }

\end{document}